\documentclass[12pt]{amsart}
\pdfoutput=1

\usepackage[bibtex, julia]{ulthiel}
\usetikzlibrary{babel}

\usepackage[utf8]{inputenc}
\usepackage[T1]{fontenc}
\usepackage{amssymb}
\usepackage{autobreak}
\usepackage[linesnumbered,ruled,vlined]{algorithm2e}
\usepackage{longtable}

\SetKwInput{KwInput}{Input}                
\SetKwInput{KwOutput}{Output}              

\usepackage{bbold}
\usepackage{enumitem}
\usepackage[]{tcolorbox}
\tcbuselibrary{listings, breakable}

\usepackage{listofitems}
\newcommand{\AnyonCode}[1]{
    \readlist*\Fvar{#1}
    \ensuremath{F^{\Fvar[1],\Fvar[2],\Fvar[3]}_{\Fvar[4],\Fvar[5],\Fvar[6],\Fvar[7]}}
}

\title[Computing the center of a fusion category]{Computing the center of a fusion category}

\author[F. Mäurer]{Fabian Mäurer}
\email{maeurer@mathematik.uni-kl.de}

\author[U. Thiel]{Ulrich Thiel}
\email{thiel@mathematik.uni-kl.de}

\address{Department of Mathematics, University of Kaiserslautern--Landau, Postfach 3049, 67663 Kaiserslautern, Germany}

\date{September 4, 2026}

\DeclareMathOperator{\ev}{ev}
\DeclareMathOperator{\coev}{coev}
\DeclareMathOperator{\FPdim}{FPdim}
\DeclareMathOperator{\Gr}{Gr}

\DeclareUnicodeCharacter{2295}{\ensuremath{\oplus}}
\DeclareUnicodeCharacter{2297}{\ensuremath{\otimes}}
\DeclareUnicodeCharacter{1D7D9}{\ensuremath{\mathbb{1}}}
\DeclareUnicodeCharacter{221A}{\ensuremath{\sqrt{}}}
\DeclareUnicodeCharacter{3C7}{\ensuremath{\chi}}
\DeclareUnicodeCharacter{22C5}{\ensuremath{\cdot}}
\DeclareUnicodeCharacter{D7}{\ensuremath{\times}}
\DeclareUnicodeCharacter{3B1}{\ensuremath{\alpha}}
\DeclareUnicodeCharacter{3C1}{\ensuremath{\rho}}
\DeclareUnicodeCharacter{2217}{\ensuremath{\ast}}
\DeclareUnicodeCharacter{3B3}{\ensuremath{\gamma}}

\usetikzlibrary{shapes.geometric}
\tikzset{
    baseline = (current bounding box.center)
}

\newcommand{\tzCanonicalBasisLeft}[7]{
    \begin{tikzpicture}[baseline = (current bounding box.center)]
        \draw (0,0) -- (0.5,1) -- node[pos = 0.5, below right = -5pt, scale = 0.7] {$#5$} (1,2) -- (1,3);
        \draw (1,0) -- (0.5,1);
        \draw (2,0) -- (1,2);
        \node[below] at (0,0) {$#1$};
        \node[below] at (1,0) {$#2$};
        \node[below] at (2,0) {$#3$};
        \node[above] at (1,3) {$#4$};
        \node[left] at (0.5,1) {$#6$};
        \node[left] at (1,2) {$#7$};
    \end{tikzpicture}
}

\newcommand{\tzCanonicalBasisRight}[7]{
    \begin{tikzpicture}[baseline = (current bounding box.center)]
        \draw (2,0) -- (1.5,1) -- node[pos = 0.5, below left = -5pt, scale = 0.7] {$#5$} (1,2) -- (1,3);
        \draw (1,0) -- (1.5,1);
        \draw (0,0) -- (1,2);
        \node[below] at (0,0) {$#1$};
        \node[below] at (1,0) {$#2$};
        \node[below] at (2,0) {$#3$};
        \node[above] at (1,3) {$#4$};
        \node[right] at (1.5,1) {$#6$};
        \node[right] at (1,2) {$#7$};
    \end{tikzpicture}
}

\newcommand{\tzBasisVector}[4 ]{
    \begin{tikzpicture}
        \draw (0,0) -- (0.5,1) -- (0.5,2);
        \draw (1,0) -- (0.5,1); 

        \node[below] at (0,0) {$#1$};
        \node[below] at (1,0) {$#2$};
        \node[above] at (0.5,2) {$#3$};
        \node[left] at (0.5,1) {$#4$};
    \end{tikzpicture}
}

\begin{document}

\begin{abstract}
 We present an algorithm for explicitly computing the categorical (Drinfeld) center of a pivotal fusion category. Our approach is based on decomposing the images of simple objects under the induction functor from the category to its center. We have implemented this algorithm in a general-purpose software framework \textsc{TensorCategories.jl} for tensor categories that we develop within the open-source computer algebra system OSCAR. We compute explicit models for the centers in form of the tuples $(X,\gamma)$ where $X$ is an object and $\gamma$ is a half-braiding. From these models we can compute the $F$-symbols and $R$-symbols. Using the data from the AnyonWiki, we were able to compute the center together with its $F$-symbols and $R$-symbols for all  the 279 multiplicity-free fusion categories up to rank 5, and furthermore some chosen examples of rank 6, including the Haagerup subfactor (presented in a separate paper). Some of these centers yield new high-rank and not multiplicity-free examples of modular categories, e.g. rank 21 and multiplicity 2.
\end{abstract}

\maketitle

\newcommand{\tageq}{\refstepcounter{equation}\tag{\theequation}}

\newcommand{\mat}[1]{\begin{bmatrix}
    #1
\end{bmatrix}}

\section{Introduction}

Let $\mathcal{C}$ be a monoidal category with associators 
\begin{equation} \label{equ:ass}
    a_{X,Y,Z} \colon (X \otimes Y) \otimes Z \overset{\simeq}{\longrightarrow} X \otimes (Y \otimes Z) \;.
\end{equation}
Throughout, we will use conventions as in the standard reference \cite{EGNO}. We are concerned with the \emph{categorical (Drinfeld) center} $\mathcal{Z}(\mathcal{C})$ of $\mathcal{C}$. This is a categorification of the notion of the center of a monoid where, as usual, equalities are replaced by isomorphisms, and we keep track of these. Specifically, a \emph{half-braiding} for an object $X \in \mathcal C$ is a natural isomorphism 
\begin{equation} 
    \gamma_X = \{\gamma_X(Y) \colon X \otimes Y \overset{\simeq}{\longrightarrow} Y \otimes X \mid Y \in \mathcal C \} 
\end{equation}
such that the diagram
\begin{equation}
    \begin{tikzcd}[font = \small, column sep = large]
        & (Y \otimes X) \otimes Z \ar[r, "a_{Y,X,Z}"] & Y \otimes (X \otimes Z) \ar[dr, "\id_Y \otimes \gamma_X(Z)"] & \\
    (X \otimes Y) \otimes Z \ar[ur, "\gamma_X(Y) \otimes \id_Z"] \ar[dr, "a_{X,Y,Z}"] & & & Y \otimes (Z \otimes X) \\
        & X \otimes (Y \otimes Z) \ar[r, "\gamma_X(Y \otimes Z)"] & (Y \otimes Z) \otimes X \ar[ur, "a_{Y,Z,X}"]  &
    \end{tikzcd}
\end{equation}
commutes for all $Y,Z \in \mathcal C$ and $\gamma_{1} = \id_X$. The center $\mathcal{Z}(\mathcal{C})$ is the category whose objects are pairs $(X,\gamma_X)$ of an object $X \in \mathcal C$ and a half-braiding $\gamma_X$ for $X$, and the morphisms $(X,\gamma_X) \to (Y,\gamma_Y)$ are morphisms $f \colon X \to Y$ in $\mathcal{C}$ such that the diagram
\begin{equation}
    \begin{tikzcd}[column sep = large]
            X \otimes Z \ar[r, "f \otimes \id_Z"] \ar[d, "\gamma_X(Z)"] & Y \otimes Z \ar[d, "\gamma_Y(Z)"] \\
            Z \otimes X \ar[r, "\id_Z \otimes f"] & Z \otimes Y
    \end{tikzcd}  
    \label{CenterMorphism}
\end{equation}
commutes for all $Z \in \mathcal C$. The category $\mathcal{Z}(\mathcal{C})$ inherits a natural monoidal structure from~$\mathcal{C}$, see \cite[Section 7.13]{EGNO}. Moreover, the half-braidings induce the structure of a braiding (``commutativity constraint'') on $\mathcal{Z}(\mathcal{C})$, so that $\mathcal{Z}(\mathcal{C})$ is a braided monoidal category, see \cite[Section 8.5]{EGNO}. Braided monoidal categories play an important role in knot theory and physics \cite{Bakalov2001}, which is one of the reasons why the center construction is important.

Among monoidal categories, fusion categories \cite{Etingof2005} play a major role since they arise in several areas of mathematics and physics. The center of a fusion category is a braided fusion category (``Müger's theorem''), in fact a modular category \cite{Bakalov2001}, and is thus of particular interest in, e.g., (2+1)-dimensional topological quantum field theory \cite{RT-Invariants}. \\

In this paper we will present an algorithm for computing the center of a pivotal fusion category (Section \ref{sec:algorithm}). The algorithm is theoretically deterministic in the sense that every step is constructive and all tools used along are in theory deterministic. 
In the process we need to find simple subobjects of objects which amounts to decomposing its endomorphism algebra. Solving this problem over a general field is hard, but there are Las Vegas algorithms which perform very well in practice, namely the MeatAxe algorithm and its variants \cite{Parker1984, Holt1994, Lux2010}. 

Our algorithm is \emph{theoretically} straightforward. The main point is that we realized it in practice. To this end, we have begun developing a general-purpose software framework \textsc{TensorCategories.jl} \cite{Maeurer_TensorCategories_jl} for tensor categories and categorical representation theory. This framework is based on the open-source computer algebra system OSCAR \cite{OSCAR, OSCAR-book} which incorporates and extends powerful systems like GAP \cite{GAP4} and Singular \cite{Singular}. OSCAR in turn is based on the high-performance programming language Julia \cite{Bezanson2017,JuliaWebsite} which, due its type system design and multiple dispatch paradigm, is nicely suited for working with categorical structures. Julia has a simple high-level syntax like Python but is much faster. OSCAR provides, e.g., fast algebra over number fields, which is crucial for our purposes. \\

As with group representations, the theory of fusion categories shows its greatest development so far over an algebraically closed field of characteristic zero. An important aspect of our algorithm and implementation is that it works over any field. This leads to interesting new results, one of which we want to mention here.  

Let $\mathcal{C}'$ be the famous Ising fusion category. This is a fusion category over the complex numbers with three simple objects denoted by $\mathbb 1$, $\chi$, and $X$. The multiplication is given by $\chi \otimes \chi = 1$, $\chi \otimes X = X \otimes \chi = X$ and $X \otimes X = \mathbb 1 \oplus \chi$, The Ising category is a special case of Tambara--Yamagami fusion categories \cite{TAMBARA1998692}. From general theory it is known \cite[Proposition 4.1]{Gelaki2009} that the center $\mathcal{Z}(\mathcal{C}')$ has 9 simple objects. The non-trivial associators of~$\mathcal{C}'$ are given by 
\begin{align*}
    a_{\chi, X, \chi} &= (-1)\id_{X} \\
    a_{X,\chi,X} &= \id_{\mathbb 1}\oplus (-1)\id_\chi \\
    a_{X,X,X} &= \frac{1}{\sqrt{2}}\begin{pmatrix}1 & 1 \\ 1 & -1\end{pmatrix}\id_{2X} \;,
\end{align*}
see \cite{TAMBARA1998692}. It follows that $\mathcal{C}'$ can be defined over $\mathbb{Q}(\sqrt{2})$, i.e., it has a $\mathbb{Q}(\sqrt{2})$-rational form~$\mathcal{C}$. This is a fusion category over $\mathbb{Q}(\sqrt{2})$, and it is known from general theory that $\mathcal{Z}(\mathcal{C})$ is a semisimple tensor category---more precisely, it is what we call a weak pre-modular category, see Section \ref{sec:assumptions}. What is the structure of $\mathcal{Z}(\mathcal{C})$? As far as we are aware, this does not appear in the literature yet. Via explicit calculations with our software in Section \ref{sec:software} we will prove:

\begin{theorem} \label{thm:ising}
    Let $\mathcal{C}$ be the Ising fusion category considered over $\mathbb{Q}(\sqrt 2)$. Then $\mathcal Z(\mathcal C)$ has five simple objects with dimensions $1,1,2,2$ and $4\cdot \sqrt 2$. The multiplication table is given by 
    
    \medskip
    \begin{center}
        \begin{tabular}{c||c|c|c|c|c}
            $\otimes$ & $\mathbb 1$ & $\overline{\mathbb 1}$ & $\chi^2$ & $\mathbb 1 + \chi $ & $X^4$ \\ \hline \hline 
            $\mathbb 1$ & $\mathbb 1$ & $\overline{\mathbb 1}$ & $\chi^2$ & $\mathbb 1 + \chi $ & $X^4$ \\ \hline
            $\overline{\mathbb 1}$ & $\overline{\mathbb 1}$ & $\mathbb 1$ &  $\chi^2$ & $\mathbb 1 + \chi $ & $X^4$ \\ \hline
            $\chi^2$ & $\chi^2$ & $\chi^2$ & $2\mathbb 1 + 2\overline{\mathbb 1}$ & $2(\mathbb 1 + \chi)$ & $2 X^4$ \\ \hline
            $\mathbb 1 + \chi $ & $\mathbb 1 + \chi $ & $\mathbb 1 + \chi $ & $2(\mathbb 1 + \chi)$ & $\mathbb 1 + \overline{\mathbb 1} + \chi^2$ & $2X^4$ \\ \hline
            $X^4$ & $X^4$ & $X^4$ & $2X^4$ & $2X^4$ & $4\mathbb 1 + 4\overline{\mathbb 1} + 4\chi^2 + 8(\mathbb 1 + \chi)$
        \end{tabular} 
    \end{center} \medskip
    and the $S$-matrix is given by 
    \begin{align*}
        S = \begin{pmatrix}
            1 & 1 & 2 & 2 & 4\cdot \sqrt 2 \\
            1 & 1 & 2 & 2 & -4 \cdot \sqrt 2 \\
            2 & 2 & 4 & -4 & 0 \\
            2 & 2 & -4 & 0 & 0 \\
            4\cdot \sqrt 2 & -4 \cdot \sqrt 2 & 0 & 0 & 0
        \end{pmatrix}
    \end{align*}
    In particular, $\mathcal{Z}(\mathcal{C})$ is a non-split modular category. It splits over $\mathbb Q(\xi_{16})$ where $\xi_{16}$ is a primitive 16-th root of unity. 
\end{theorem} 

We emphasize that our computational construction of the center is completely explicit so that we obtain the simple central objects as objects of $\mathcal{C}$ together with half-braidings (omitted here to save space). We can also analyze the splitting of the central objects when passing to the algebraic closure of $\mathbb{Q}$ or finite field extensions of $\mathbb{Q}$, see Section \ref{sec:software}. For split categories the computation of $F$-symbols and $R$-symbols is then only a computational task we also implemented, see section \ref{sec:F_symbols}. Our software allows both exact computations and numeric (floating point) computations. \\

In \cite{vercleyen2024lowrankmultiplicityfreefusioncategories, AnyonWiki}, $F$-symbols for all multiplicity-free fusion categories up to rank seven were computed. These are available via the Mathematica package Anyonica \cite{Anyonica} and the website AnyonWiki \cite{AnyonWiki} by Gert Vercleyen. Thanks to Gert Vercleyen we could integrate this database into our package and apply our implementation. As a main result of this paper we computed the centers of all 279 multiplicity-free fusion categories up to rank five. Some of these centers yield new high-rank and not multiplicity-free examples of modular categories, e.g. rank 21 and multiplicity 2. This will be useful, e.g., to check conjectures in the future. For the resulting braided categories we stored $F$-symbols and $R$-symbols in a database \cite{TensorCategoriesDatabase}. This database can be accessed directly from within \textsc{TensorCategories.jl}. All data can be reproduced with a script in the folder \verb+computations/center_paper+ of \cite{Maeurer_TensorCategories_jl}, which also contains a table of timings for each category. In sections \ref{sec:software} and \ref{sec:AnyonCenters} we show how to work with \textsc{TensorCategories.jl} and how to access the data. 

Moreover, with our software we were able to compute the center of the fusion category coming from the Haagerup subfactor. We will present this in the separate paper \cite{HaagerupCenter} since the computation required new techniques to find a model of the Haagerup fusion category where the $F$-symbols live in a sufficiently small field. 

\vspace{\baselineskip}

In Section \ref{sec:assumptions} we will specify our setting and make precise what we mean by ``computing the center''. In Section \ref{sec:naive} we discuss the computation of half-braidings to illustrate the complexity of the problem. In Section \ref{sec:algorithm} we discuss our algorithm. In Section \ref{sec:splitting} we derive a theoretical result on the splitting of central objects under field extension. In Section \ref{sec:software} we present our software framework and illustrate it with the proof of Theorem \ref{thm:ising}. Finally, in Section \ref{sec:AnyonCenters} we will present the data for the centers of all multiplicity-free fusion categories up to rank five.

\subsection*{Acknowledgments}
This work is a contribution to the SFB-TRR 195 ``Symbolic Tools in Mathematics and their Application'' of the German Research Foundation (DFG), Project-ID 28623755. We thank Gert Vercleyen for the effort of porting the fusion categories from Anyonica to our package. We thank Liam Rogel for testing our software and providing helpful feedback. We thank Gunter Malle for comments on a preliminary version of this paper.

\tableofcontents


\section{Basic assumptions, notation, and facts} \label{sec:assumptions}

We first want to emphasize that we do \emph{not} require monoidal categories to be strict, i.e. the associators \eqref{equ:ass} may be non-identity morphisms. It is true that by Mac Lane's strictness theorem \cite[Theorem 2.8.5]{EGNO}, any monoidal category is monoidally equivalent to a strict monoidal category. But such a monoidal equivalence may be difficult to represent in the computer, which is against our algorithmic approach. Moreover, to represent a monoidal category in the computer one will usually choose a skeleton. Even though every category is equivalent to a skeletal category, a monoidal category may in general not be monoidally equivalent to a skeletal strict monoidal category, i.e. one cannot achieve strict and skeletal simultaneously, see \cite[Remark 2.8.7]{EGNO}. We will thus incorporate associators everywhere and will not necessarily assume the category to be strict.

We do assume, however, strictness for the unit object $\mathbb 1$ of a monoidal category $\mathcal{C}$ in the sense that we require the unit morphism $\iota\colon \mathbb 1 \otimes \mathbb 1 \to \mathbb 1$ to be the identity, and for any $X \in \mathcal{C}$ we identify $\mathbb 1 \otimes X$ and $X\otimes \mathbb 1$ with $X$. This is a basic assumption also in \cite{EGNO}, see \cite[Remark 2.2.9]{EGNO}, and it is not problematic as was strictness because one can simply choose the skeleton appropriately, see \cite[Exercise 2.8.8]{EGNO}. 

We let $\Bbbk$ be a field. We do \emph{not} require $\Bbbk$ to be algebraically closed, and we do not make any assumptions on its characteristic. This level of generality is also addressed in \cite[Section 4.16]{EGNO}. Even though we can work with an algebraic closure of $\mathbb{Q}$, say, on a computer algebra system like OSCAR, the arithmetic will in most cases be slower than with a number field. It will thus be of advantage to work over smaller fields. Moreover, this will lead us to interesting rationality questions. Since terminology in the theory of tensor categories over general fields is in parts not yet entirely consistent throughout the literature, we will be very precise about our definitions below. In general, our basis will be \cite{EGNO}.

\subsection{Fusion categories}

We propose the following generalization of \cite[Definition 4.1.1]{EGNO} to an arbitrary base field. By a \emph{weak fusion category} over $\Bbbk$ we mean a $\Bbbk$-linear abelian semisimple rigid monoidal category $\mathcal{C}$ which has only finitely many isomorphism classes of simple objects, the unit object is simple, the Hom-spaces are finite-dimensional over $\Bbbk$, and the bifunctor $\otimes$ is $\Bbbk$-bilinear.\footnote{Such a category is finite in the sense of \cite[Definition 1.8.6]{EGNO}, the requirement of enough projectives following from semisimplicity. Moreover, the simplicity of the unit object implies that the category is indecomposable as required in \cite[Definition 4.1.1]{EGNO}.} The Grothendieck ring $\Gr(\mathcal{C})$ of such a category is a \emph{weak fusion ring} as in \cite[Section 3.8]{EGNO}, see \cite[Proposition 4.2]{Etingof2012}, hence our terminology. 

We call an object $X \in \mathcal{C}$ \emph{scalar} if $\Bbbk \simeq \End_{\mathcal{C}}(X)$ via the canonical map. We say that $\mathcal{C}$ \emph{splits} if all simple objects of $\mathcal{C}$ are scalar, and then we call $\mathcal{C}$ a \emph{fusion category}. In case $\Bbbk$ is algebraically closed, this definition coincides with \cite[Definition 4.1.1]{EGNO}. Also, for a general field $\Bbbk$, it coincides with the definition in \cite[Section 1.G]{bruguieres2013center}, this paper being important for us since it generalizes many results and constructions from algebraically closed fields to arbitrary base fields (and rings). We want to incorporate the non-split case (with the notion of \emph{weak}) since the center may be non-split, even if we start with a split category.

\subsection{Coends}

If $\mathcal{C}$ is any category and $F\colon \mathcal C \times \mathcal C^{op} \to \mathcal C$ is a functor, then a \emph{coend} of $F$ is an object $C \in \mathcal{C}$ together with morphisms $c_X \colon F(X,X) \to C$ for each $X \in \mathcal{C}$ which are universal with the property that the diagram 
\begin{equation}
\begin{tikzcd}
    F(X,Y) \ar[r, "F{(f, \id_Y)}"] \ar[d, "F{(\id_X, f)}"] & F(Y,Y) \ar[d, "c_Y"] \\
    F(X,X) \ar[r, "c_X"] & C 
\end{tikzcd}
\end{equation}
commutes for each morphism $f\colon X \to Y \in \mathcal C$. The coend object $C$, if it exists, is unique up to isomorphism. We refer to \cite[Section IX.6]{MacLane1971} fore more details on coends. 

If $\mathcal C$ is a rigid monoidal category, then if the functor $\mathcal C\times \mathcal C^{op} \to \mathcal C$ defined by $(X,Y) \mapsto  X \otimes Y^\ast$ and $(f,g) \mapsto f \otimes g^\ast$, has a coend, it is called a (canonical) \emph{coend} of $\mathcal C$. It follows from \cite[Theorem 3.4]{Shimizu2017} that any weak fusion category $\mathcal{C}$ has a coend. If~$\mathcal{C}$ is a fusion category, a coend $C$ is explicitly given by
\begin{equation} \label{coend_for_fusion}
C = \bigoplus_{i=1}^n X_i \otimes X_i^* \;,
\end{equation}
where $X_1,\ldots,X_n$ is a complete set of representatives of the simple objects of $\mathcal{C}$, see \cite[Section 3D]{bruguieres2013center}.

\subsection{Dimensions}

If $\mathcal{C}$ is a pivotal monoidal category, we have the notions of \emph{left dimension} $\mathrm{dim}_l(X)$ and \emph{right dimension} $\mathrm{dim}_r(X)$ for any object $X \in \mathcal{C}$ given by, see \cite[Section 4.7]{EGNO}. These dimensions are elements of $\End_{\mathcal{C}}(\mathbb 1)$. 
If~$\mathcal{C}$ admits a coend $C$ and if $\mathcal  C$ is spherical, we set 
\begin{equation}
    \dim_L(\mathcal{C}) \coloneqq \dim_L(C) ,~~\dim_R(\mathcal C) \coloneqq \dim_R(C)
\end{equation}
and call this the (categorical) \emph{left (right) dimension} of~$\mathcal{C}$. This definition is as in \cite[Section 3D]{bruguieres2013center}. If~$\mathcal{C}$ is a pivotal fusion category, it follows from \eqref{coend_for_fusion} and the multiplicativity of dimension, \cite[Proposition 4.7.12]{EGNO}, that 
\begin{equation} \label{eq:dim_squares}
\dim(\mathcal C) \coloneqq \dim_L(\mathcal{C}) = \dim_L(C) = \sum_{i=1}^n \vert X_i \vert \in \End_{\mathcal{C}}(\mathbb 1) \simeq \Bbbk 
\end{equation}
where $\vert X_i \vert = \dim_L(X_i) \dim_R(X_i)$ is the squared norm of $X_i$, see \cite[Section 2.3]{Etingof2005}.
Our definition of $\dim(\mathcal{C})$ thus agrees with \cite[Definition 7.21.3]{EGNO}. 

\subsection{Modular categories} \label{sec:modular}

By a \emph{pre-modular (weak) fusion category} we mean a braided spherical (weak) fusion category. To such a category $\mathcal{C}$ we associate a matrix $S$ defined as follows: if $X_1,\ldots,X_n$ are the simple objects of $\mathcal{C}$ up to isomorphism, then $S$ is the square matrix $S$ of size $n$ with entries
\begin{equation}
S_{ij} = \mathrm{Tr}(c_{X_j X_i} c_{X_i X_j}) \in \Bbbk \;.
\end{equation}
Here, $c_{X_i X_j}$ denotes the braiding $X_i \otimes X_j \to X_j \otimes X_i$ and $\mathrm{Tr}$ denotes the trace \cite[Equation 8.40]{EGNO}, which relies on the spherical structure. We call~$S$ the \emph{$S$-matrix} of $\mathcal{C}$. We say that a pre-modular category is \emph{modular} if its $S$-matrix is invertible. Our definitions of pre-modular, the $S$-matrix, and modular are the same as in \cite{EGNO} (Definitions 8.13.1, 8.13.2, and 8.14.1), but we note that in \cite{EGNO} it is assumed that $\Bbbk$ is algebraically closed of characteristic zero. We nonetheless think they make sense in general, and they will be convenient for stating some facts. In \cite[Section 3E]{bruguieres2013center} the notion of a modular category over a general field (or ring) is defined which involves a non-degeneracy condition for a pairing defined using the coend. A modular category as defined here is modular as defined in \cite{bruguieres2013center} by \cite[Remark 3.4]{bruguieres2013center}.

\subsection{Structure of the center}

The key fact about the center we use here is the following.

\begin{theorem}[{Bruguières--Virelizier \cite{bruguieres2013center}}] \label{center_main_theorem}
    Let $\mathcal{C}$ be a pivotal fusion category over a field~$\Bbbk$ with $\dim(\mathcal{C}) \neq 0$. Then $\mathcal{Z}(\mathcal{C})$ is a pre-modular weak fusion category and 
    \begin{equation} \label{eq:center_dimension}
        \dim(\mathcal{Z}(\mathcal{C})) = \dim(\mathcal{C})^2 \;.
    \end{equation}
If $\mathcal{Z}(\mathcal{C})$ splits, it is a modular fusion category.
\end{theorem}

Let us make some remarks on this fact. It is a deep conjecture \cite[Conjecture 2.8]{Etingof2005} that any fusion category over an algebraically closed field of characteristic zero is pivotal. As far as we are aware, all currently known fusion categories admit a spherical structure. We note, however, that a given pivotal structure may not be spherical, see \cite[Example 4.59]{Tubbenhauer22}. If $\Bbbk$ is of characteristic zero, the condition $\dim(\mathcal{C}) \neq 0$ is always true by \cite[Theorem 2.3]{Etingof2005}.\footnote{In \cite{Etingof2005} the base field is assumed to be algebraically closed. But our definition of a fusion category requires the category to be split, and then we can argue by scalar extension.} All the relevant structures on $\mathcal{C}$ (direct sums, linearity, duals, pivotal, spherical) naturally transfer to $\mathcal{Z}(\mathcal{C})$, see \cite[Section 1J]{bruguieres2013center} and \cite[Theorem 3.6]{Shimizu2017}. The only property left to establish that $\mathcal{Z}(\mathcal{C})$ is a weak fusion category is semisimplicity, and this is indeed not true in general, see \cite[Example 2.5]{bruguieres2013center}. In fact, it is proven in \cite[Corollary 2.3]{bruguieres2013center} that $\mathcal{Z}(\mathcal{C})$ is semisimple if and only if $\dim(\mathcal{C}) \neq 0$. So, if $\dim(\mathcal{C}) \neq 0$, it follows that $\mathcal{Z}(\mathcal{C})$ is a weak pre-modular fusion category. The relation $\dim(\mathcal{Z}(\mathcal{C})) =  \dim(\mathcal{C})^2$ was shown in \cite[Theorem 2.1]{bruguieres2013center}. Finally, if $\mathcal{Z}(\mathcal{C})$ splits, it follows from \cite[Theorem 2.1]{bruguieres2013center} and our remarks in Section \ref{sec:modular} that it is a modular category.

We note that in general it is not true that $\mathcal{Z}(\mathcal{C})$ splits: using our software we will show in Section \ref{sec:software} that the center of the Ising category over $\mathbb{Q}(\sqrt{2})$ is a non-split weak modular category. We note that the phenomenon of non-splitting of the center of a split fusion category is known, see e.g. \cite{morrison2012non}. We hope that our software will help to explore further examples of this exciting phenomenon. Finally, we note that the above theorem in case of an algebraically closed field was established much earlier by Müger \cite{subfactors}. The upshot of \cite{bruguieres2013center} was to establish it for any field.

\subsection{\texorpdfstring{$F$}{F}-symbols and \texorpdfstring{$R$}{R}-symbols} \label{sec:F_symbols}


Let $\mathcal C$ be a fusion category with simple objects $X_1,...,X_r$. Consider the multiplicity spaces $H_{i,j}^k := \Hom(X_i \otimes X_j, X_k)$ and fix a basis for every combination. We represent such a basis morphism $f \in H_{i,j}^k$ as a string diagram 

\begin{align}
    f = \tzBasisVector{X_i}{X_j}{X_k}{f}
\end{align}
and hence there are canonical bases for $\Hom(X_i \otimes (X_j \otimes X_k), X_l)$ and $\Hom((X_i \otimes X_j) \otimes X_k, X_l)$ given by morphisms of the form 

\begin{align}
     \tzCanonicalBasisRight{X_i}{X_j}{X_k}{X_l}{X_n}{\gamma}{\delta}~~\text{and}~~\tzCanonicalBasisLeft{X_i}{X_j}{X_k}{X_l}{X_m}{\alpha}{\beta}. \label{eq:canonical_bases}
\end{align}
The associator $a_{X_i,X_j,X_k}$ induces an isomorphism on these spaces by 

\begin{align}
    \begin{aligned}
    \Hom(X_i \otimes (X_j \otimes X_k), X_l) &\to \Hom((X_i \otimes X_j) \otimes X_k, X_l) \\f &\mapsto f \circ a_{X_i,X_j,X_k}
    \end{aligned}
\end{align}
which is just a linear transformation, i.e. we can express it by 

\begin{align}
    \tzCanonicalBasisRight{X_i}{X_j}{X_k}{X_l}{X_n}{\gamma}{\delta} \circ a_{X_i,X_j,X_k} = \sum\limits_{m = 1}^r \sum\limits_{\substack{\alpha \in H_{i,j}^m \\ \beta \in H_{m,k}^l}} \left[F_{i,j,k}^l\right]_{n,\gamma,\delta}^{m,\alpha,\beta} \tzCanonicalBasisLeft{X_i}{X_j}{X_k}{X_l}{X_m}{\alpha}{\beta} \;. \label{equ:f-matrices}
\end{align}
The matrices $F_{i,j,k}^l$ are called the \emph{$F$-matrices} and their entries are called \emph{$F$-symbols} of $\mathcal C$. Similarly, the braiding isomorphisms $\sigma_{X,Y}\colon X \otimes Y \to Y \otimes X$ give rise to a basis transformation 

\begin{align}
    \begin{aligned}
    \Hom(X_j \otimes X_i, X_k) &\to \Hom(X_i \otimes X_j, X_k) \\ f &\mapsto f \circ \sigma_{X_i,X_j} 
    \end{aligned}
\end{align}
which again can be expressed in the bases as 

\begin{align}
    \tzBasisVector{X_j}{X_i}{X_k}{\beta} \circ \sigma_{X_i,X_j} = \sum\limits_{\alpha \in H_{i,j}^k} \left[R_{i,j}^k\right]_\beta^\alpha \tzBasisVector{X_i}{X_j}{X_k}{\alpha} \;. \label{equ:r-matrices}
\end{align}
Here the entries of $R_{i,j}^k$ are referred to as \emph{$R$-symbols}. 

\begin{remark}
Equations \eqref{equ:f-matrices} and \eqref{equ:r-matrices} are the composition-dual
projection-space forms of the \(F\)- and \(R\)-moves in
\cite[Equations~(2.14) and~(2.30)--(2.32)]{BondersonShtengelSlingerland2008}. The entries are identical.
\end{remark}

\begin{remark}
    Note that we started with a fusion category, i.e. a split semisimple category. This is necessary for the $F$-symbols to be scalars. If we allowed non-split (weak fusion) categories, the resulting entries would be morphisms in $\End(X_i)$, i.e. elements of some simple algebra over $\mathbb k$. 
\end{remark}


\subsection{Computing the center}

We can now make precise what we mean by ``computing the center'': given a spherical fusion category $\mathcal{C}$ over a field $\Bbbk$ with $\dim(\mathcal{C}) \neq 0$, we want to determine a complete set of representatives of the simple objects of $\mathcal{Z}(\mathcal{C})$. Since $\mathcal{Z}(\mathcal{C})$ is semisimple and the extra structure on $\mathcal{Z}(\mathcal{C})$ like associator is inherited from $\mathcal{C}$, this gives complete knowledge of $\mathcal{Z}(\mathcal{C})$. The resulting objects then are given---by their very definition---in the form of tuples $(X,\gamma)$ where $X \in \mathcal C$ and $\gamma$ is a half-braiding. With this representation we can compute explicit tensor products, direct sums and Hom-spaces.  

In our computational perspective that we take on this problem we actually want to achieve more. First, to work explicitly with such a category $\mathcal{C}$ in the computer we need to assume we have some sort of finite encoding (model) of $\mathcal{C}$. One, but not the only, possibility would be via the multiplication table of the Grothendieck ring $\mathrm{Gr}(\mathcal{C})$ of $\mathcal{C}$ and the $F$-symbols of $\mathcal{C}$, which encode the associators via a collection of matrices \cite[Section 4.9]{EGNO}. Determining such a model requires hard work in most examples, but we assume it exists and start from there. Our software provides a general framework to work with models of categories, the specifics being up to the user. We have already implemented models of several standard examples, e.g. graded vector spaces, representation categories for finite groups, Verlinde type categories and some more to find in the documentation \cite{Maeurer_TensorCategories_jl}. Moreover, all the fusion categories from the AnyonWiki \cite{AnyonWiki} are available. We can now formulate our precise goal.

\begin{tcolorbox}
    \textbf{Goal}. We want to compute a model of $\mathcal{Z}(\mathcal{C})$ in terms of the model of the original category $\mathcal{C}$, avoiding abstract equivalences and the like. In particular, we want to compute the simple objects of $\mathcal{Z}(\mathcal{C})$ explicitly as pairs of an object together with a half-braiding, such that the $F$-symbols and $R$-symbols can be computed.
\end{tcolorbox}

\begin{example}
Let us illustrate our goal for a simple example. Let $G$ be a finite group and let $\Bbbk$-$\mathsf{Vec}_G$ be the category of finite-dimensional $G$-graded vector spaces over $\Bbbk$, see \cite[Example 2.3.6]{EGNO}. This is a fusion category with simple objects $\{ \delta_g \}_{g \in G}$, where~$\delta_g$ is the 1-dimensional vector space with degree $g$. Moreover, it admits a natural spherical structure. Here, the associators are trivial. There are versions $\Bbbk$-$\mathsf{Vec}_G^\omega$  with non-trivial associator as well, see \cite[Remark 2.6.2]{EGNO}. It follows from Equation \eqref{eq:dim_squares} that $\dim(\Bbbk$-$\mathsf{Vec}_G) = |G| \cdot 1_{\Bbbk}$. The center $\mathcal{Z}(\Bbbk$-$\mathsf{Vec}_G)$ is well-known \cite[Example 8.5.4]{EGNO}: the simple objects are parametrized by pairs $(C,V)$, where $C$ is a conjugacy class in $G$ and $V$ is an irreducible finite-dimensional representation of the centralizer of an element $g \in C$. 

But actual central objects are pairs $(X,\gamma_X)$ of a $G$-graded vector space $X$ and a half-braiding $\gamma_X$ for $X$. What do the actual simple central objects look like? Answering this is our goal, and precisely this is done by our algorithm and software implementation. We give a concrete example in Figure \ref{vecGcenter}. Note that it is very easy to model $\Bbbk$-$\mathsf{Vec}_G$ in the computer: we just need the group $G$, and we describe morphisms as certain matrices over $\Bbbk$.

\begin{figure}[htbp]
    \renewcommand*{\arraystretch}{1.5}
    \tiny\centering
    \[
    \begin{array}{c||c|c|c|c|c|c}
            & (12)  & (13) & (23)   & (123)     & (132) \\ \hline\hline
        ()  & 1     & 1    & 1      & 1         & 1     \\ \hline
        ()  & -1    & -1   & -1     & 1         & 1     \\ \hline
        2\cdot()& \mat{1&0\\-1&-1} & \mat{0&1\\1&0} & \mat{-1&-1\\0&1} & \mat{-1&-1\\1&0} & \mat{0&1\\-1&-1} \\ \hline
        (23) \oplus (12) \oplus (13) & \mat{0&0&1\\0&1&0\\1&0&0} & \mat{0&1&0\\1&0&0\\0&0&1} & \mat{1&0&\\0&0&1\\0&1&0} & \mat{0&1&0\\0&0&1\\1&0&0} & \mat{0&0&1\\1&0&0\\0&1&0} \\ \hline
        (23) \oplus (12) \oplus (13) & \mat{0&0&-1\\0&-1&0\\-1&0&0} & \mat{0&1&0\\1&0&0\\0&0&-1} & \mat{-1&0&\\0&0&1\\0&1&0} & \mat{0&-1&0\\0&0&-1\\1&0&0} & \mat{0&0&1\\-1&0&0\\0&-1&0} \\ \hline
        (132)\oplus(123) & \mat{0&1\\1&0} & \mat{0&1\\1&0} & \mat{0&1\\1&0} & \id & \id \\ \hline
        (132)\oplus(123) & \mat{0&\xi_3\\\xi_3^2&0} & \mat{0&1\\1&0} & \mat{0&\xi_3^2\\\xi_3&0} & \mat{0&\xi_3^2\\\xi_3&0} & \mat{0&\xi_3\\\xi_3^2&0} \\ \hline
        (132)\oplus(123) & \mat{0&\xi_3^2\\\xi_3&0} & \mat{0&1\\1&0} & \mat{0&\xi_3\\\xi_3^2&0} &  \mat{0&\xi_3\\\xi_3^2&0} & \mat{0&\xi_3^2\\\xi_3&0}
    \end{array}
    \]
    \caption{The actual center of $\Bbbk$-$\mathsf{Vec}_G$ for $G$ the symmetric group $S_3$ and $\Bbbk = \mathbb{Q}(\xi_3)$, where $\xi_3$ is a primitive third root of unity. The table lists the simple central objects $(X,\gamma_X)$. The first column lists the underlying object $X$ of $\Bbbk$-$\mathsf{Vec}_G$. The other columns are indexed by the simple objects $Y$ of $\Bbbk$-$\mathsf{Vec}_G$ and specify the half-braiding $\gamma_X(Y)$. The data in this table is readily provided by our software.} \label{vecGcenter}
\end{figure}

\end{example}



\section{Computing half-braidings} \label{sec:naive}

Before we discuss our algorithm and software implementation, we first want to discuss the computation of half-braidings in order to illustrate the complexity of this problem. 

Note that we have a forget functor
\begin{equation} \label{eq:forget}
    F \colon \mathcal{Z}(\mathcal{C}) \rightarrow \mathcal{C}, \quad (X,\gamma_X) \mapsto X \;.
\end{equation}
The image of the induced map $\mathrm{Gr}(\mathcal{Z}(\mathcal{C})) \to \mathrm{Gr}(\mathcal{C})$ between Grothendieck rings obviously lies in the center of $\mathrm{Gr}(\mathcal{C})$. In the example in Figure \ref{vecGcenter} we can observe two crucial peculiarities of the center:
\begin{enumerate}
\item If $(X,\gamma_X)$ is a simple object of $\mathcal{Z}(\mathcal{C})$, then $X$ is not necessarily simple in $\mathcal{C}$. 
\item If $(X,\gamma_X)$ is a simple object of $\mathcal{Z}(\mathcal{C})$ there may exist another non-isomorphic half-braiding $\gamma_X'$ for $X$, so that $(X,\gamma_X')$ is another simple object of $\mathcal{Z}(\mathcal{C})$. In other words, there may be several non-isomorphic simple objects of $\mathcal{Z}(\mathcal{C})$ lying above the same object of $\mathcal{C}$ under the forget functor.
\end{enumerate}

It seems a possible strategy to compute the center would be to inductively form direct sums $Z$ of the simple objects in $\mathcal{C}$ and determine the half-braidings on $Z$. The inductive procedure would ensure that this produces simple central objects and, over a splitting field, one could use the dimension formula \eqref{eq:center_dimension} with the coend formula \eqref{coend_for_fusion} to establish that one found all simple central objects. However, we will now argue that the determination of half-braidings and their isomorphism classes as objects of the center seems computationally infeasible. 

First, semisimplicity of $\mathcal{C}$ allows one to encode a half-braiding $\gamma_Z$ by a finite amount of information, namely by the maps $\gamma_Z(X_i)$ for the simple objects $X_i$ of $\mathcal{C}$. We used this already in the example in Figure \ref{vecGcenter}. The precise statement is as follows.

\begin{lemma}[{Müger \cite[Lemma 3.3]{subfactors}}]\label{discreteCondition}
	Let $\mathcal C$ be a fusion category with simple objects $\{X_i\}$. Let $Z\in\mathcal C$. There is a bijection between half-braidings for $Z$ and families of morphisms $\{\gamma_{Z}(X_i) \in \Hom_{\mathcal{C}}(Z\otimes X_i, X_i\otimes Z)\}$ such that for all $i,j,k$ and $t\in \Hom_{\mathcal{C}}(X_k,X_i\otimes X_j)$ the diagram 
    \begin{equation}
	\begin{tikzcd}[column sep = large]
		Z\otimes X_k \ar[d, "\id_Z\otimes t" left] \ar[r, "\gamma_Z(X_k)"] & X_k\otimes Z \ar[r, "t \otimes \id_Z"] & (X_i\otimes X_j) \otimes Z \ar[d, "a_{X_i,X_j,Z}"]\\
		Z\otimes (X_i\otimes X_j) \ar[d, "a_{Z,X_i,X_j}^{-1}"] 	&  & X_i \otimes (X_j\otimes Z) \\
        (Z\otimes X_i)\otimes X_j \ar[r, "\gamma_Z(X_i) \otimes \id_{X_j}"] & (X_i\otimes Z) \otimes X_j \ar[r, "a_{X_i,Z,X_j}"] & X_i \otimes (Z\otimes X_j) \ar[u, "\id_{X_i} \otimes \gamma_Z(X_j)"]
	\end{tikzcd}
    \end{equation}
    commutes and $\gamma_Z(\mathbb 1) = \id_Z$.
\end{lemma}

\begin{remark}
This statement is proven in \cite{subfactors} more generally for a semisimple monoidal category. In \cite{subfactors} the base field is assumed to be algebraically closed but this does not play a role in the proof. In \cite{subfactors} monoidal categories are assumed to be strict. We have given the non-strict version here. The statement and its proof are straightforward but there is one important detail which is less obvious: half-braidings need to be isomorphisms but we just consider morphisms $\gamma_Z(X_i)$ without requiring invertibility. The fact that we automatically get isomorphisms is due to the existence of dual objects in $\mathcal{C}$ (rigidity) and \cite[Lemma 2.2]{subfactors}.
\end{remark}

Lemma \ref{discreteCondition} allows us to set up a system of equations for half-braidings $\gamma_Z$ for $Z \in \mathcal{C}$. Let $X_i,X_j,X_k$ be simple objects of $\mathcal{C}$ and let $t \in \Hom_{\mathcal{C}}(X_k, X_i \otimes X_j)$. Then the lemma states an equation 
\begin{equation}
   \phi(\gamma_Z(X_k),t) = \psi(\gamma_Z(X_i),\gamma_Z(X_j),t) \;,
\end{equation}
where 
\begin{equation}
    \phi(\gamma_Z(X_k),t) =  a_{X_i,X_j,X_k} \circ (t \otimes \id_Z) \circ \gamma_Z(X_k)
\end{equation}
and
\begin{equation}
    \psi(\gamma_Z(X_i),\gamma_Z(X_j),t) = (\id_{X_i} \otimes \gamma_Z(X_j)) \circ a_{X_i,Z,X_j} \circ (\gamma_Z{X_i} \otimes \id_{X_j}) \circ a_{Z,X_i,X_j}^{-1} \circ (\id_Z \otimes t) \;.
\end{equation}

Clearly $\phi$ and $\psi$ are linear in $t$ as well as in $\gamma_Z(X_i),\gamma_Z(X_j)$, and $\gamma_Z(X_k)$ respectively. Thus, whenever the equations hold for a basis of $\Hom_{\mathcal{C}}(X_k, X_i \otimes X_j)$, they hold for all $t$. After choosing a basis $f_1,...,f_r$ for $\Hom(Z \otimes X_k, X_i \otimes (Z \otimes X_j))$ and bases $g_1^l,...,g_{r_l}^l$ for $\Hom_{\mathcal{C}}(Z\otimes X_l, X_l \otimes Z)$ we can replace $\gamma_Z(X_l)$ with 
\begin{equation}
    \gamma_Z(X_l) = a_{1}^lg_{1}^l + \cdots + a_{r_l}^lg_{r_l}^l
\end{equation}
and get a system of \emph{quadratic} equations by comparing coefficients. Writing out the full details is technical, but we have implemented this system in our software framework and the exact details can be found in the source code. Using algebraic solvers like msolve \cite{msolve} this, in principle, yields an approach to computing the half-braidings for a fixed object. However, as we illustrate in Example \ref{ex:half_braidings_ideal} below, even in very small examples, the systems are extremely complicated and usually form a positive-dimensional ideal, so that there are infinitely many solutions and msolve cannot be used. If the base field is not a finite field or the field of rational numbers, the computational situation is even worse. Also, even if we have found (a parametric form of) the solutions, we still need to determine their isomorphism classes as objects of the center. Hence, without further simplifications from theory (which we do not know) this approach seems infeasible. 

\begin{example} \label{ex:half_braidings_ideal}
Consider the category $\Bbbk$-$\mathsf{Vec}_{G}$ for $G = S_3$ and $\Bbbk = \QQ$. Consider the object $Z = \mathbb 1 \oplus \mathbb 1$. We can see in Figure \ref{vecGcenter}, which we computed using our algorithm from the next section, that there is a half-braiding for $Z$. The situation is actually more delicate: there is a (up to isomorphism unique) half-braiding for $Z$ as a \emph{simple} object in the center, but there is also a half-braiding coming from the decomposition $ \mathbb 1 \oplus \mathbb 1$. Note that the half-braidings are  defined over $\QQ$, we do not need the bigger field $\QQ(\xi_3)$ as in Figure \ref{vecGcenter}. The ideal defined by the equations for half-braidings for $Z$ is given in Figure \ref{fig:half_braidings_ideal}. We computed this with the function \verb+TensorCategories.build_center_ideal+. The ideal is generated by 104 polynomials in $24$ indeterminates. Its dimension is 3, even though, as we noted, there is up to isomorphism only \emph{one} half-braiding for $Z$ as a simple object of the center. 
\begin{figure}[htbp]
\begin{align*} 
\begin{autobreak}
\langle
 x_{1} - x_{5}^2 - x_{6} x_{7},
  x_{2} - x_{5} x_{6} - x_{6} x_{8},
  x_{3} - x_{5} x_{7} - x_{7} x_{8},
  x_{4} - x_{6} x_{7} - x_{8}^2,
  ....
\rangle
\end{autobreak} \\
= 
\begin{autobreak}
    \langle
x_{7}^2 + x_{7}x_{11} - x_{11}^2 + x_{19}^2 + x_{7}x_{23} - x_{11}x_{23} - x_{23}^2, ...
    \rangle
\end{autobreak}

\end{align*}
\caption{The ideal for half-braidings for the object $\mathbb 1 \oplus \mathbb 1$. Only a few elements are given here to give an impression. The second expression is a Gröbner basis with respect to the degree reversed lexicographic ordering. It still has 85 generators.} \label{fig:half_braidings_ideal}
\end{figure} 
\end{example}

\section{An algorithm based on the induction functor} \label{sec:algorithm}

We will now present our algorithm for computing the center. 

\begin{tcolorbox}
\begin{assumption}
Throughout, we assume that~$\mathcal{C}$ is a pivotal fusion category over a field $\Bbbk$ with $\dim(\mathcal{C}) \neq 0$.
\end{assumption}
\end{tcolorbox}

We denote by $X_1,\ldots,X_n$ a complete set of the isomorphism classes of simple objects of~$\mathcal{C}$. We denote by $\psi \colon (-) \to (-)^{\ast \ast}$ the pivotal structure on $\mathcal{C}$. For an object $X$ we denote by $\ev_X \colon X^\ast \otimes X \to \mathbb 1$ and $\coev_X \colon \mathbb 1 \to X \otimes X^*$ the evaluation and coevaluation morphisms, respectively.

We note that by \cite[Theorem 4.8.4]{EGNO} we have $\dim(X_i) \neq 0$. In the proof of this fact in \cite{EGNO} the base field is assumed to be algebraically closed, but the proof remains valid for a split simple object, which is the case in our definition of a fusion category.

\subsection{The induction functor and the basic idea} 
Recall the forget functor $F \colon \mathcal{Z}(\mathcal{C}) \to \mathcal{C}$ from Equation \ref{eq:forget}. It admits a two-sided adjoint 
\begin{equation} 
    I\colon \mathcal C \to \mathcal Z(\mathcal C) \;,
\end{equation}
called the \emph{induction functor} \cite[Section 9.2]{EGNO}. Explicitly, for $V \in \mathcal{C}$ the underlying object of $I(V)$ is given by
\begin{equation} \label{eq:I_obj}
    FI(V) = \bigoplus \limits_{i=1}^n (X_i \otimes V) \otimes X_i^\ast
\end{equation}
and the half-braiding for this object is given by the component morphisms 
\begin{equation} \label{eq:I_half_braiding}
    \gamma(Z)_{ij} = \sum\limits_{f \in B, g \in B'}
    \begin{tikzcd}[column sep = huge]
        ((X_i \otimes V) \otimes X_i^\ast) \otimes W \ar[r, "a_{X_i \otimes V, X_i^\ast, W}"] & (X_i \otimes V) \otimes (X_i^\ast \otimes W) \ar[d, "(f \otimes \id_V) \otimes g"] \\
        & ((W \otimes X_j) \otimes V) \otimes X_j^\ast  \ar[d, "a_{W,X_j,V} \otimes \id_{X_j^\ast}"] \\
        W \otimes ((X_j \otimes V) \otimes X_j^\ast) & (W \otimes (X_j \otimes V)) \otimes X_j^\ast \ar[l, "a_{W, X_j \otimes V, X_j^\ast}"] 
    \end{tikzcd} \;,
\end{equation}
where $B$ and $B'$ are a pair of dual bases of $\Hom_{\mathcal{C}}(X_i, W\otimes X_j)$ and $\Hom_{\mathcal{C}}(W\otimes X_j, X_i) \cong \Hom_{\mathcal{C}}(X_i^\ast\otimes W, X_j^\ast)$ where the isomorphism is given by the natural isomorphism of the dual adjunction. These details can be found \cite[Section 2]{bruguieres2013center}. We have added the associators here. \\

We want to give the explicit adjunction isomorphisms between the forget functor and the induction functor. To this end, let us assume here that $X_1 = \mathbb 1$. The left adjunction isomorphisms given by
\begin{equation}
    \begin{aligned}\label{adjunction_isomorphism}
    \Hom_{\mathcal C}(V,Y) &\cong \Hom_{\mathcal Z(\mathcal C)}(I(V), (Y,\gamma)) \\    
    f &\mapsto \sum\limits_{i} \dim_R X_i \cdot \phi_i(f) \circ p_i
    \end{aligned}
\end{equation}
where $p_i\colon \bigoplus_{i} (X_i \otimes V) \otimes X_i^\ast \to (X_i \otimes V) \otimes X_i^\ast$ are the projections and $\phi_i(f)$ is given by
\begin{equation}
   \begin{tikzcd}[column sep = huge]
        (X_i \otimes V) \otimes X_i^\ast \ar[d, "{:= \phi_i(f)}"] \ar[r, "\id_{X_i} \otimes f \otimes \id_{X_i^\ast}"] & (X_i \otimes Y) \otimes X_i^\ast \ar[r, "\gamma(X_i)^{-1} \otimes \id_{X_i^\ast}"] & (Y \otimes X_i) \otimes X_i^\ast \ar[d, "a_{Y,X_i,X_i^\ast}"] \\
        Y & Y \otimes (X_i^{\ast\ast} \otimes X_i^\ast) \ar[l, "\id_Y \otimes \ev_{X_i^\ast}"]   &  Y \otimes (X_i \otimes X_i^\ast) \ar[l, "\id_{Y} \otimes \psi_{X_i} \otimes \id_{X_i^\ast}"]
    \end{tikzcd}
\end{equation}
and an inverse $\Hom_{\mathcal Z(\mathcal C)}(I(V), (Y,\gamma)) \to \Hom_{\mathcal C}(V,Y)$ is given by $g \mapsto p_1 \circ g$. The proof is analogous to \cite[Theorem 2.3]{kirillov2010turaev} where a different pairing is used. The right adjunction isomorphisms are given by 
\begin{equation}
    \begin{aligned}
        \label{right_adjunction_isomorphism}
        \Hom_{\mathcal C}(Y, V) &\cong \Hom_{\mathcal Z(\mathcal C)}((Y,\gamma), I(V)) \\
        f &\mapsto \sum\limits_i \iota_i \circ  \sigma_i(f)
    \end{aligned}
\end{equation}
where $\iota_i \colon X_i \otimes V \otimes X_i^\ast \to FI(V)$ are the inclusions and $\sigma_i(f)$ is defined by
\begin{equation}
    \begin{tikzcd}
        Y \ar[r, "\id_Y \otimes \coev_{X_i}"] \ar[d, ":= \sigma_i(f)"] & Y \otimes (X_i \otimes X_i^\ast) \ar[r, "a_{Y,X_i,X_i^\ast}^{-1}"] & (Y\otimes X_i) \otimes X_i^\ast \ar[d, "\gamma(X_i)\otimes \id_{X_i^\ast}"] \\
        (X_i \otimes V) \otimes X_i^\ast & & (X_i \otimes Y) \otimes X_i^\ast \ar[ll, "\id_{X_i} \otimes f \otimes \id_{X_i^\ast}"]
    \end{tikzcd}
\end{equation}
and an inverse is given by $\Hom_{\mathcal C}(Y, V) \to \Hom_{\mathcal Z(\mathcal C)}((Y,\gamma), I(V)) \colon g \mapsto \iota_1 \circ g$. Again the proof follows the exact same arguments as the left adjoint case. \\

At the heart of our algorithm lies the following well-known property of the induction functor which is an immediate consequence of the adjunction. 

\begin{lemma}
Every simple object $(Z, \gamma_Z) \in \mathcal Z(\mathcal C)$ arises as a subobject of $I(X_i)$ for some simple object $X_i \in \mathcal{C}$. 
\end{lemma}

\begin{proof}
Since $F( (Z,\gamma_Z) )$ is a non-zero object of $\mathcal{C}$, it has a simple quotient $X_j$ and so there is a non-zero morphism $F((Z,\gamma_Z) ) \to X_j$. Hence, using the adjunction, we conclude
\begin{equation}
0 \neq \Hom_{\mathcal C}(F((Z,\gamma_Z)), X_j) \simeq \Hom_{\mathcal Z(\mathcal C)}((Z,\gamma_Z), I(X_j)).
\end{equation}
We conclude there is a non-zero morphism $f \colon (Z,\gamma_Z) \to I(X_j)$. Since $(Z,\gamma_Z)$ is simple, the kernel of $f$ must be zero, so $f$ is a monomorphism and therefore $(Z,\gamma_Z)$ is a subobject of $I(X_j)$.
\end{proof}

In other words, the induction functor $I$ is \emph{surjective} as in \cite[Definition 1.8.3]{EGNO}, see also \cite[Corollary 7.13.11]{EGNO}. The lemma yields an obvious idea how to compute a complete set $\mathrm{Irr}(\mathcal{Z}(\mathcal{C}))$ of the isomorphism classes of simple objects of $\mathcal{Z}(\mathcal{C})$. We formalize this in Algorithm~\ref{alg:center_first_version}.

\begin{algorithm}[htbp]
    \DontPrintSemicolon
    \caption{Compute $\mathrm{Irr}(\mathcal{Z}(\mathcal{C}))$ -- First version \label{alg:center_first_version}}
    Initialize $\mathrm{Irr}(\mathcal{Z}(\mathcal{C})) = \emptyset$ \;
    \For{$i=1$ \KwTo $n$}{
        Compute the object $I(X_i)$ \;
        Compute a complete set $\mathrm{Irr}(I(X_i))$ of isomorphism classes of simple subobjects of $I(X_i)$ \;
        \For{$(Z,\gamma) \in \mathrm{Irr}(I(X_i))$}{
            \If{$(Z,\gamma)$ \textup{is not isomorphic to any object in} $\mathrm{Irr}(\mathcal{Z}(\mathcal{C}))$}
            {
                Add $(Z,\gamma)$ to $\mathrm{Irr}(\mathcal{Z}(\mathcal{C}))$ \;
            } 
        }
    }
    \Return $\mathrm{Irr}(\mathcal{Z}(C))$
\end{algorithm}

Using the explicit formulas in Equations \eqref{eq:I_obj} and \eqref{eq:I_half_braiding}, we were able to implement the induction functor in our software framework, so Step 3 in Algorithm \ref{alg:center_first_version} is constructive. The problems are Step 4 (the computation of simple subobjects) and Step 6 (the isomorphism check). Decomposing a semisimple object $X \in \mathcal C$ in an abelian category is equivalent to decomposing the endomorphism algebra $\End_{\mathcal{C}}(X)$, i.e. if 
$$X = \bigoplus\limits_{i=1}^k n_i\cdot X_i$$ for simple objects $X_i$ then 
$$\End_{\mathcal{C}}(X) = \bigoplus\limits_{i=1}^k \mathrm{Mat}_{n_i\times n_i} \End_{\mathcal{C}}(X_i)$$
and $X_i \cong \mathrm{Im}(f_i)$ for any $0 \neq f_i \in \End_{\mathcal{C}}(X_i)$ considered as an element of $\End_{\mathcal{C}}(X_i)$. The decomposition of algebras is in general a very hard task and there is no algorithm known. In our implementation we use the MeatAxe algorithm \cite{Parker1984, Holt1994, Lux2010}, enhanced by various further technical tweaks. Details can be found in the source code of our implementation. In practice, for small algebras, this works reasonably well. Even for complicated fusion categories like the Haagerup subfactor or $PSU(2)_k$ the algebras in consideration are small (rank < 20). 

In order to decompose endomorphism algebras at all, we need to be able to determine these algebras explicitly. More generally, we need to be able to compute Hom-spaces in the center. We will discuss this in Section \ref{Morphisms}. First, we will discuss a simple optimization of Algorithm \ref{alg:center_first_version}.

\subsection{Center-generating simple objects}

In most examples the computation of $I(X_i)$ for all simple $X_i$ will result in redundant computations: 1) it might happen that $I(X_i) \cong I(X_j)$ for $X_i \not \cong X_j$; 2) there are examples where all simple objects are contained in just a few inductions. This leads us to the following concept.

\begin{definition}
    A \emph{set of center-generating simples} for $\mathcal Z(\mathcal C)$ is a collection of simple objects $S_1,\ldots,S_l \in \mathcal C$ such that every simple object $Z\in \mathcal Z(\mathcal C)$ is a subquotient of at least one $I(S_i)$.
\end{definition}

Now, the goal is to (algorithmically) find a small set of center-generating simples. Here is one approach. Recall that an object $X \in \mathcal C$ is called \emph{invertible} if $X \otimes X^\ast \cong \mathbb 1$. This immediately implies that $X$ is simple. We will write $\mathrm{Inv}(\mathcal C) \subseteq \mathrm{Irr}(\mathcal C)$ for the invertible objects in $\mathcal C$.

\begin{lemma}\label{isomorphic_inductions}
    Let $V,W \in \mathcal C$ such that $W \cong J \otimes V \otimes J^\ast$ for some $J \in \mathrm{Inv}(\mathcal C)$. Then $I(V) \cong I(W)$.
\end{lemma}

\begin{proof}
    We can assume that $\mathcal{C}$ is strict. Since $J$ is invertible, the object $X_i \otimes J$ is simple and $- \otimes J$ is an auto-equivalence of $\mathcal{C}$. The half-braidings on the underlying object of $I(W) \cong I(J \otimes V \otimes J^\ast)$ are given by the components
    $$\gamma(Z)_{ij} = \sum\limits_{f \in B, g \in B'} f \otimes \id_J \otimes \id_V \otimes \id_{J^\ast} \otimes g.$$
    The $f\otimes\id_J$ and $\id_{J^\ast} \otimes g$ form again a pair of dual basis. Hence, $I(V)$ and $I(J \otimes V \otimes J^\ast)$ carry the same half-braiding up to reordering of the summands. 
\end{proof}

The action of $\mathrm{Inv}(\mathcal C)$ defines an equivalence relation on $\mathrm{Irr}(\mathcal C)$. So, by Lemma \ref{isomorphic_inductions} the representatives for the equivalence classes form a center-generating set. This set does not need to be minimal but it is one that we can algorithmically compute. \\

Using the concept of center-generating simples we give in Algorithm \ref{alg1} a refined version of Algorithm \ref{alg:center_first_version}. We restricted the computations to a set of center-generating simples and replaced the condition for whether a simple object is accepted as new. The latter works since for every simple object $X_i \in \mathrm{Irr}(\mathcal C)$ and simple object $Z \in \mathrm{Irr}(\mathcal Z(\mathcal C))$ we have $$\Hom_{\mathcal C}(X_i, F(Z)) \cong \Hom_{\mathcal Z(\mathcal C)}(I(X_i), Z)$$ and therefore whenever $\Hom_{\mathcal C}(X_i, Z) \neq 0$ for any $X_i$ processed before, then $Z$ is already in the list of simple objects of $\mathcal Z(\mathcal C)$.

\begin{algorithm}[ht]
    \caption{Compute $\mathrm{Irr}(\mathcal{Z}(\mathcal{C}))$ -- Refined version \label{alg1}}
    \DontPrintSemicolon
    \KwInput{A fusion category $\mathcal C$}
    \KwOutput{A complete set $\mathrm{Irr}(\mathcal{Z}(\mathcal C))$ of isomorphism classes of simple objects of $\mathcal{Z}(\mathcal C)$}
    Initialize $\mathrm{Irr}(\mathcal{Z}(\mathcal{C})) = \emptyset$ \;

    Compute a set $S = \{S_1,...,S_l\} \subseteq \mathrm{Irr}(\mathcal C)$ of center-generating simples, e.g., via computing the conjugacy classes on $\mathrm{Irr}(\mathcal C)$\;

    \For {$i \in \{1,...,l\}$} {
        Compute $I(S_i)$\;
        Compute all non-isomorphic simple subobjects $Z_1,...,Z_k$ of $I(S_i)$\;
        \For {$j \in \{1,...,k\}$} {
            \If {$\Hom_{\mathcal{C}}(Z_j, S_t) = 0$ \textup{for all} $t \in \{1,...,i-1\}$}{
                Add $Z_j$ to $\mathrm{Irr}(\mathcal{Z}(\mathcal C))$
            }
        }
    }
    \Return $\mathrm{Irr}(\mathcal{Z}(C))$
\end{algorithm}

\subsection{Computing Hom-spaces in $\mathcal Z(\mathcal C)$} \label{Morphisms}

There are multiple ways to obtain morphisms in the center of a fusion category $\mathcal C$. The first and most straight forward approach is to solve for the condition \eqref{CenterMorphism}. If $f \in \Hom_{\mathcal Z(\mathcal C)}(X,Y) \subset \Hom_{\mathcal C}(F(X),F(Y))$  we can write \[f = a_1g_1 + \cdots + a_k g_k\] where $g_1,...,g_k$ is a basis of $\Hom_{\mathcal C}(F(X),F(Y))$. Condition \eqref{CenterMorphism} yields now a set of equations \[\gamma_Y(Z) \circ \left( f \otimes \id_Z \right) = \left( \id_Z \otimes f \right) \circ \gamma_X(Z)\]
for each simple object $Z \in \mathcal C$ which are linear in the $a_i$. By solving this system we obtain a basis of the space $\Hom_{\mathcal Z(\mathcal C)}(X,Y)$. This approach scales badly since $k$ grows quickly with the number of direct summands in $F(X)$ and $F(Y)$, and the setup of the system is expensive. Though for ``small'' objects it can outperform the following approaches. Especially if the objects are simple, i.e. if we want to decide isomorphy between simples, this is the way to go.\\

In general a much better approach is to take advantage of the adjunction isomorphisms \ref{adjunction_isomorphism} and \ref{right_adjunction_isomorphism}. With those we immediately obtain endomorphism spaces of all objects of the form $I(X)$ since 
\begin{equation} \label{eq:ind_hom}
    \Hom_{\mathcal Z(\mathcal C)}(I(X),I(X)) \cong \Hom_{\mathcal C}(X, FI(X))
\end{equation}
and with some effort also arbitrary Hom-spaces as follows. Let $X,Y \in \mathcal Z(\mathcal C)$ and denote by $Z_1,...,Z_r$ the simple objects of $\mathcal Z(\mathcal C)$. Since $\mathcal Z (\mathcal C)$ is semisimple we can express any morphism $f\colon X \to Y$ as a sum of morphisms $X \to Z_i \to Y$. If $\lbrace S_1,\ldots,S_l \rbrace$ is a set of center-generating simples, then for all $Z_i$ there is some $j$ such that there is a non-zero morphism $Z_i \hookrightarrow I(S_j)$. Then $\Hom_{\mathcal{Z}(\mathcal{C})}(X,Y)$ is spanned by morphisms of the form $X \to I(S_j) \to Y$. We summarize this approach in Algorithm \ref{alg:hom_spaces}. 

\begin{algorithm}[!ht]
    \DontPrintSemicolon
    \caption{Computing Hom-spaces in $\mathcal Z(\mathcal C)$ \label{alg:hom_spaces}}
    \KwInput{Objects $X,Y \in \mathcal Z(\mathcal C)$. A set $\{S_1,...,S_l\}$ of center-generating simples}
    \KwOutput{A basis of $\Hom_{\mathcal{Z}(\mathcal{C})}((X, \gamma_X),~(Y,\gamma_Y))$}
    Compute a basis of the spaces $\Hom_{\mathcal{Z}(\mathcal{C})}(I(S_j),Y)$ and $\Hom_{\mathcal{Z}(\mathcal{C})}(X,I(S_j))$ for all $j = 1,...,l$ via the adjunction \eqref{eq:ind_hom} \;

    Collect the set of maps $M_j = \{p \circ i \in \Hom_{\mathcal{Z}(\mathcal{C})}(X,Y) \mid p \in \Hom_{\mathcal{Z}(\mathcal{C})}(I(S_j),Y), i \in \Hom_{\mathcal{Z}(\mathcal{C})}(X, I(S_j))\}$ for all $j = 1,...,l$\;

    Use linear algebra to find a basis in the union of the $M_j$
\end{algorithm}

Still another way is to project onto $\Hom_{\mathcal Z(\mathcal C)}(X,Y) \subseteq \Hom_{\mathcal C}(F(X),F(Y))$ with the formula in the following lemma. In practice, however, a combination of linear systems of equations and algorithm \ref{alg:hom_spaces} is the most efficient way to get to morphisms.

\begin{lemma}[{\cite[Lemma 2.2.]{kirillov2010turaev}}]\label{HomProjection}
    For $(X,\gamma_X), (Y,\gamma_Y) \in \mathcal Z(\mathcal C)$ the map $E_{X,Y}: \Hom_{\mathcal C}(X,Y) \to \Hom_{\mathcal C}(X,Y)$ given by 
	\[E_{X,Y}(t) = \frac{1}{\dim \mathcal C} \sum\limits_{i =1 }^n \dim X_i \phi_i(t) \;, \]
	where $\phi_i(t)$ is given by 
	\[
	\begin{tikzcd}[column sep = huge]
		X \ar[r, "\id_X\otimes \coev(X_i)"] \ar[dddd, "\phi_i(t)"]& X \otimes (X_i \otimes X_i^\ast) \ar[r, "a_{X,X_i,X_i^\ast}^{-1}"] & (X \otimes X_i) \otimes X_i^\ast \ar[d, "\gamma_X(X_i)\otimes \id_{X_i^\ast}"]\\
        & & (X_i\otimes X)\otimes X_i^\ast \ar[d, "\id_{X_i}\otimes t\otimes \id_{X_i^\ast}"]   \\
		 & & (X_i\otimes Y)\otimes X_i^\ast \ar[d, "a_{X_i,Y,X_i^\ast}"] \\
         & & X_i\otimes (Y\otimes X_i^\ast) \ar[d, "\psi_{X_i}\otimes \gamma_Y(X_i^\ast)"] \\
		 Y & (X_i^{\ast\ast} \otimes X_i^\ast) \otimes Y \ar[l, "\ev_{X_i^\ast}\otimes \id{Y}"] & X_i^{\ast\ast} \otimes (X_i^\ast \otimes Y) \ar[l, "a_{X_i^{\ast\ast},X_i^\ast,Y}^{-1}"] 
	\end{tikzcd}\] 
	is a projection from $\Hom_{\mathcal C}(X,Y)$ onto $\Hom_{\mathcal Z(\mathcal C)}((X,\gamma_X), (Y,\gamma_Y))$.
\end{lemma}

\section{Splitting of simple central objects} \label{sec:splitting}

In this section we want to take a closer look at some properties of the center when the ground field $\kk$ is not algebraically closed. 

\subsection{Scalar extension}

For the following concepts we also refer to \cite{Etingof2012} and \cite{morrison2012non}.
%
Let $\mathcal C$ be a $\mathbb k$-linear semisimple abelian category and $\kk \hookrightarrow \mathbb{K}$ be a field extension. We define the category $\mathcal C \otimes \mathbb{K}$ to have the same objects as $\mathcal C$ and morphism spaces are given by 
    \begin{equation}
        \Hom_{\mathcal C \otimes \mathbb K}(X,Y) = \Hom_{\mathcal C}(X,Y) \otimes_{\kk} \mathbb K\;.
    \end{equation}

The category $\mathcal C \otimes \mathbb K$ is again $\mathbb k$-linear and also $\mathbb K$-linear. However, it may not be abelian if $\mathcal{C}$ does not split since the scalar extension of an endomorphism algebra might gain new idempotents which are not split (i.e. the object they would project to does not exist). Consider the example $\mathcal C = \mathrm{Rep}_{\mathbb Q}(C_3)$ of rational representations of the cyclic group of order $3$. There is a 2-dimensional representation with endomorphism ring isomorphic to $\mathbb K = \QQ(\xi_3)$, where $\xi_3$ is a third root of unity. Tensoring with $\QQ(\xi_3)$ splits the algebra and hence produces projectors onto the completely irreducible subrepresentations. Formally, these objects are not yet contained in the category $\mathcal C\otimes \mathbb K$. 

Hence, we want to consider the \emph{Karoubian closure} $\mathrm{Kar}(\mathcal C \otimes \mathbb K)$ of $\mathcal C \otimes \mathbb K$, which formally adds all ``direct summands'', i.e. the objects in $\mathrm{Kar}(\mathcal C \otimes \mathbb K)$ are pairs $(X,e)$ of an object $X \in \mathcal{C} \otimes \mathbb K$ and an idempotent $e \in \mathrm{End}_{\mathcal{C} \otimes \mathbb K}(X)$. We refer to \cite[Chapter 11]{EMTW} for basics on this construction. The category $\mathrm{Kar}( \mathcal C \otimes \mathbb K)$ is a semisimple $\mathbb K$-linear abelian category. We call it the \emph{scalar extension} of $\mathcal{C}$ to $\mathbb K$ and denote it with $\mathcal C \boxtimes \mathbb K := \mathrm{Kar(\mathcal C \otimes K)}$. If $\mathcal{C}$ has a monoidal structure, then so does $\mathcal C \boxtimes \mathbb K $. If $\mathcal{C}$ is a weak fusion category, then so is $\mathcal C \boxtimes \mathbb K $. We say $\mathbb K$ is a \emph{splitting field} for $\mathcal C$ if $\mathcal C \boxtimes \mathbb K $ is split.

Notice that we have an embedding $\mathcal{C} \otimes \mathbb{K} \hookrightarrow \mathcal{C} \boxtimes \mathbb{K}$, and the simple objects of $\mathcal{C} \boxtimes \mathbb{K}$ arise as direct summands of the images of the simple objects of $\mathcal{C} \otimes \mathbb{K}$. Moreover, if $X$ and $Y$ are non-isomorphic simple objects, then no direct summand of $X \otimes \mathbb K$ is isomorphic to a ``direct summand'' of $Y \otimes \mathbb K$ since otherwise there would be a non-zero morphism $X \to Y \in \mathcal C \boxtimes \mathbb K$, given by projection and inclusion, but this is not possible since 
\[
\Hom_{\mathcal{C} \boxtimes \mathbb K }(X, Y) =\Hom_{\mathcal{C}}(X,Y) \otimes \mathbb K = 0 \;. 
\]

\begin{remark}
    For a weak fusion category $\mathcal C$ over $\mathbb k$ and a field extension $\mathbb k \hookrightarrow \mathbb K$ the construction above is equivalent to the construction of the $\mathbb k$-linear Deligne tensor product $\mathcal C \boxtimes \mathrm{Vect}_{\mathbb K}$ \cite{lopezcocomplete}. Hence, the notation $\mathcal C \boxtimes \mathbb K = \mathcal C \boxtimes \mathrm{Vect}_{\mathbb K}$ is justified.
\end{remark}

\subsection{Computing the center over a splitting field}

Combining the results of the previous sections we give the final algorithm to compute the split simple objects of the center of fusion category, if necessary over a field extension.

\begin{algorithm}[htbp]
    \DontPrintSemicolon
    \caption{Compute the split simples $\mathrm{Irr}(\mathcal{Z}(\mathcal{C}) \boxtimes \mathbb K)$ \label{alg:center_final_version}}
    \KwInput{A fusion category $\mathcal C$}
    \KwOutput{A list of representatives for the isomorphism classes of simple objects in $\mathcal Z(\mathcal C) \boxtimes \mathbb K$ for a splitting field $\mathbb K$}
    Compute the simples of $\mathcal Z(\mathcal C)$\;
    Determine a common splitting field $\mathbb K$ for all endomorphism algebras of simples in $\mathcal Z(\mathcal C)$\;
    Compute all simple subobjects of $Z$ over $\mathbb K$ for all simple $Z \in \mathcal Z(\mathcal C)$\;
    \Return A list of representatives for the isomorphism classes
\end{algorithm}
Line 2 might seem mystical at first, but is just a matter of computing the endomorphism algebras. Computing a splitting field for those is performed by OSCAR. Often, the field that splits the algebras is contained in a cyclotomic extension of the base field. For performance reasons these seem to be preferable to minimal splitting fields.

Next we will give a straightforward method to obtain the $F$-symbols of the resulting fusion category summarized in algorithm \ref{alg:F_symbols}. The $R$-symbols are analogous. 

\begin{algorithm}[htbp] 
    \DontPrintSemicolon
    \caption{Compute the $F$-symbols of a fusion category $\mathcal C$ \label{alg:F_symbols}}
    \KwInput{A fusion category $\mathcal C$ with simples $X_1,...,X_r$}
    \KwOutput{The $F$-symbols of $\mathcal C$}
    Compute the multiplicity spaces $\Hom_{\mathcal C}(X_i \otimes X_j, X_k)$ for all $i,j,k \in \{1,...,r\}$ and fix bases\;
    Build the canonical bases \eqref{eq:canonical_bases} of the spaces $\Hom_{\mathcal C}((X_i \otimes X_j) \otimes X_k, X_l)$ and $\Hom_{\mathcal C}(X_i \otimes (X_j \otimes X_k), X_l)$ for all $i,j,k,l \in \{1,...,r\}$\;
    Precompose the basis elements of $\Hom_{\mathcal C}(X_i \otimes (X_j \otimes X_k), X_l)$ with the associator $a_{X_i,X_j,X_k}$ and express the results in terms of the basis of $\Hom_{\mathcal C}((X_i \otimes X_j) \otimes X_k, X_l)$\;
    \Return The coefficients obtained in the previous step as $F$-symbols
\end{algorithm}

\subsection{Frobenius--Perron dimension}

From now on, let $\mathcal C$ be a weak fusion category. The Frobenius--Perron dimension of an object $X \in \mathcal{C}$ is defined to be the largest real eigenvalue of the matrix of the action of the isomorphism class of $X$ on the Grothendieck ring $\mathrm{Gr}(\mathcal{C})$ of $\mathcal{C}$ by multiplication, see \cite[Section 4.5]{EGNO}. In case $\mathbb k$ is algebraically closed, the Frobenius--Perron dimension of the category $\mathcal{C}$ itself is defined in \cite[Definition 6.1.7]{EGNO} as
\begin{equation}
    \FPdim(\mathcal C) := \sum\limits_{i = 1}^n \FPdim(X_i)^2 \;.
\end{equation}
In \cite{non-split-fusion} this definition was extended to a general base field $\mathbb k$ via
\begin{equation} 
    \FPdim(\mathcal C) := \sum\limits_{i = 1}^n \frac{\FPdim(X_i)^2}{\dim_{\End_{\mathcal{C}}(\mathbb 1)} \End_{\mathcal{C}}(X_i)} = \sum_{i=1}^n \frac{\FPdim(X_i)^2}{\dim_{\mathbb k} \End_{\mathcal{C}}(X_i)} \;,
\end{equation}
where we use our assumption that $\mathbb 1$ is split simple. In analogy with Equation \eqref{eq:center_dimension} it is proven in \cite[Theorem 4.1.8]{non-split-fusion} that
\begin{equation} \label{eq:fpdim_center}
    \FPdim(\mathcal Z(\mathcal C)) = \FPdim(\mathcal C)^2 \;.
\end{equation}
This formula implies nice behavior of the simple central objects under extension to a splitting field.

\begin{proposition}\label{simplesdecomposition}
    Let $\mathcal{C}$ be a fusion category. Let $\overline \kk$ be an algebraic closure of $\kk$ and let $\overline{\mathcal{Z}(\mathcal{C})} = \mathcal{Z}(\mathcal{C}) \boxtimes \overline{\kk}$. If $Z \in \mathcal Z(\mathcal C)$ is a simple object, then all direct summands of $Z \otimes \overline \kk$ in  $\overline{\mathcal{Z}(\mathcal{C})}$ occur with the same multiplicity and all have the same Frobenius--Perron dimension.
\end{proposition}

\begin{proof}
    Let $Z_1,\ldots,Z_r$ be the non-isomorphic simple objects of $\mathcal{Z}(\mathcal{C})$. Let 
    \begin{equation}
        Z_i \otimes \overline \kk = a_{i,1}Z_{i,1} \oplus \cdots \oplus a_{i,n_i} Z_{i,n_i}
    \end{equation}
    be the decomposition of $Z_i \otimes \overline \kk$ in $\overline{\mathcal{Z}(\mathcal{C})}$ into pairwise non-isomorphic simple objects $Z_{i,j}$. By our remarks above, the $Z_{i,j}$ are pairwise non-isomorphic and yield all the simple objects of $\overline{\mathcal{Z}(\mathcal{C})}$. For any $i$ the Cauchy--Schwarz inequality yields
   \begin{equation} \label{csequal}
    \left(\sum\limits_{j=1}^{n_i} a_{i,j} \cdot \FPdim(Z_{i,j})\right)^2 \leq \left(\sum\limits_{j=1}^{n_i} \FPdim(Z_{i,j})^2\right)\left(\sum\limits_{j=1}^{n_i} a_{i,j}^2\right) \;.
   \end{equation}
   This is equivalent to
   \begin{equation}
     \frac{\left(\sum\limits_{j=1}^{n_i} a_{i,j} \cdot \FPdim(Z_{i,j})\right)^2}{\sum\limits_{j=1}^{n_i} a_{i,j}^2} \leq \sum\limits_{j=1}^{n_i} \FPdim(Z_{i,j})^2 \;. 
   \end{equation}
   Notice that
   \begin{equation}
    \sum_{j=1}^{n_i} a_{i,j}^2 = \dim_{\overline{\kk}} \mathrm{End}_{\overline{\mathcal{Z}(\mathcal{C})}}(Z_i \otimes \overline \kk) = \dim_{\kk} \End_{\mathcal{Z}(\mathcal{C})}( Z_i )\;.
   \end{equation}
   Moreover, since the Frobenius--Perron dimension is a ring morphism by \cite[Proposition 3.3.6]{EGNO}, we have
   \begin{equation}
    \FPdim( Z_i \otimes \overline \kk ) = \sum_{j=1}^{n_i} a_{i,j} \FPdim(Z_{i,j}) \;.
   \end{equation}
   Hence, $\FPdim(- \otimes \overline \kk)$ can be considered as a character of the Grothendieck ring of $\mathcal{Z}(\mathcal{C})$ taking positive values on the isomorphism classes of simple objects. Since this property uniquely characterizes the Frobenius--Perron dimension by \cite[Proposition 3.3.6]{EGNO}, it follows that
   \begin{equation}
     \FPdim(Z_i) = \FPdim( Z_i \otimes \overline \kk ) \;.
   \end{equation}
   Now, 
   \begin{align}
    \FPdim(\overline{\mathcal{Z}(\mathcal{C})}) & = \sum_{i=1}^r \sum_{j=1}^{n_i} \FPdim(Z_{i,j})^2 \geq \sum_{i=1}^r \frac{\left(\sum\limits_{j=1}^{n_i} a_{i,j} \cdot \FPdim(Z_{i,j})\right)^2}{\sum\limits_{j=1}^{n_i} a_{i,j}^2} \\
    & = \sum_{i=1}^r \frac{\FPdim(Z_i)^2}{\dim_{\kk} \End_{\mathcal{Z}(\mathcal{C})}( Z_i )} = \FPdim(\mathcal{Z}(\mathcal{C})) \;.
   \end{align}
   Since $\mathcal{C}$ is split, it follows that $\mathcal{C} \otimes \overline \kk$ is semisimple, hence equal to its Karoubian closure, so $\mathcal{C}$ is a ``rational form'' of $\mathcal{C} \otimes \overline \kk$ in the terminology of Theorem \cite[Section 2]{morrison2012non}. It thus follows from \cite[Lemma 5.1]{morrison2012non} that $\mathcal{Z}(\mathcal{C})$ is a rational form of $\mathcal{Z}(\mathcal{C} \otimes \overline \kk)$, i.e.,
   \begin{equation}
    \mathcal{Z}(\mathcal C \otimes \overline \kk) = \overline{\mathcal{Z}(\mathcal{C})} \;.
   \end{equation}
   Using Equation \eqref{eq:fpdim_center}, we thus get
   \begin{equation}
    \FPdim( \mathcal{Z}(\mathcal{C}) ) = \FPdim(\mathcal{C})^2 = \FPdim(\mathcal C \otimes \overline \kk)^2 = \FPdim( \mathcal{Z}(\mathcal C \otimes \overline \kk)) = \FPdim(\overline{\mathcal{Z}(\mathcal{C})} ) \;.
   \end{equation}
   We thus have equality in the Cauchy--Schwarz inequality \eqref{csequal}. This implies that there is some $\lambda \in \mathbb{R}$ such that
   \begin{equation}
     \FPdim(Z_{i,j}) = \lambda a_{i,j} 
   \end{equation}
   for all $i,j$. This holds if and only if 
   \begin{equation} \label{eq:fpdim_ai}
    \frac{a_{i,j}}{\FPdim(Z_{i,j})} = \lambda = \frac{a_{i',j'}}{\FPdim(Z_{i',j'})}
   \end{equation}
    for all $i,j,i',j'$. 
   
   Now, fix some $i$. Consider $A = \End_{\mathcal{Z}(\mathcal{C})}{Z}_i$. By assumption, $A$ is a simple division algebra. Let $Z(A)$ be the center of $A$. Consider the extension $Z(A)\otimes_{\kk} A$ which is as $Z(A)$-algebra isomorphic to $\bigoplus_{j=1}^{[Z(A):\kk]} A$. Finally, $A$ is central simple as $Z(A)$-algebra and hence has a splitting field $\mathbb K \supseteq \mathbb k$ such that $A \cong \mathrm{Mat}_{l \times l}(\mathbb K)$. We conclude that $\mathbb K \otimes A \cong \bigoplus_{j = 1}^{[Z(A):\kk]} \mathrm{Mat}_{l\times l}(\mathbb K)$ as $\mathbb K$-algebras. This implies that $Z \otimes \mathbb K$ decomposes in $\mathcal{Z}(\mathcal{C}) \boxtimes \mathbb K$ into $[ Z(A) : \kk]$ non-isomorphic simple objects, each with multiplicity $l$.  This forces $a_{i,j} = l$ for all $j$, and thus, by Equation \eqref{eq:fpdim_ai}, all $Z_{i,j}$ have the same Frobenius--Perron dimension.
\end{proof}

\begin{remark}
    If $\kk$ is a finite field the statement above is even stronger since all finite dimensional division algebras over a finite field are fields. Thus, a simple object $Z$ decomposes into non-isomorphic simple objects with multiplicity one.
\end{remark}

\begin{example}
    The case that a simple object occurs with multiplicity greater than one from Lemma \ref{simplesdecomposition} can occur. Let $\mathbf Q$ be the quaternion group and consider the category $\mathrm{Vec}_{\QQ}(\mathbf Q)$ of finite-dimensional $\mathbf Q$-graded vector spaces over $\QQ$. Then there is a central object lying over $4 \cdot \mathbf 1_{\mathbf Q}$ corresponding to the four-dimensional irreducible $\QQ$-representation of $\mathbf Q$. This object has (similar to the representation) an endomorphism ring isomorphic to the rational quaternions $\mathbb H_{\QQ}$. Analogously to the representation it decomposes into two copies of the same simple object over $\QQ(\sqrt{-1})$.
\end{example}

\section{Software framework and examples} \label{sec:software}

Our software \cite{Maeurer_TensorCategories_jl} is completely open-source and comes with a documentation. Its general purpose is to serve as a framework for experiments with tensor categories and categorical representation theory. Instead of going into long discussions about its design and capabilities---for this we refer to the documentation---we illustrate it by working through the example from the introduction: the Ising fusion category over $\mathbb Q(\sqrt 2)$. This example is also reproduced in the documentation, where it will be easier to copy and paste the code.

For users new to Julia we note Julia uses just-in-time compilation (JIT). This is, in parts, what makes Julia so fast but it means that the \emph{first} execution of a function always takes a bit of time (since its code will be compiled)---afterward it is faster. In a running Julia session our package can be installed (automatically with all dependencies like OSCAR) and loaded as follows:

\begin{tcolorbox}[title = \juliasymbol, breakable]
    \begin{juliaprompt}
julia> import Pkg
julia> Pkg.add("TensorCategories")
julia> using TensorCategories, Oscar
    \end{juliaprompt} \vspace{-\baselineskip}
\end{tcolorbox}

\subsection{The Ising Category over $\mathbb{Q}(\sqrt{2})$}

Users are completely free in how exactly they want to model a category. For encoding a fusion category given in terms of a fusion ring and explicit associators ($F$-matrices) we provide an own convenient structure called \texttt{SixJCategory}\footnote{The name \texttt{SixJCategory} origins in the classical name for the $F$-symbols used in for example \cite{EGNO}. In general $6j$-symbol can be used synonymously to $F$-symbol.}.

\begin{remark}
For a \texttt{SixJCategory}, the functions \texttt{F\_symbols} and
\texttt{R\_symbols} return the \(F\)- and \(R\)-symbols as dictionaries.
Since version 0.7.1, two dictionary layouts are available through the
keyword \texttt{convention}.

With \texttt{convention=:bonderson}, the returned dictionaries satisfy
\[
\begin{aligned}
 \mathsf D_F^{\mathrm B}
 [[i,j,k,l,m,\alpha,\beta,n,\gamma,\delta]]
 &=
 \left[F^l_{i,j,k}\right]^{m,\alpha,\beta}_{n,\gamma,\delta},\\
 \mathsf D_R^{\mathrm B}[[i,j,k,\alpha,\beta]]
 &=
 \left[R^k_{i,j}\right]^\alpha_\beta .
\end{aligned}
\]

The default \texttt{convention=:column\_major\_packing} is a consistent
programmatic packing. It reads each \(F\)-matrix down its columns, as in
Julia's \texttt{vec}, and assigns the entries successively to the admissible
keys
\[
 [i,j,k,l,n,\delta,\gamma,m,\alpha,\beta]
\]
traversed in the nested order
\[
 (m,n,\beta,\alpha,\gamma,\delta),
\]
with \(m\) outermost and \(\delta\) innermost. In the multiplicity-free
case, let \(m_1,\ldots,m_q\) and \(n_1,\ldots,n_q\) be the ordered indices
labelling the internal simple objects in the left- and right-associated
fusion trees of Equation \eqref{equ:f-matrices}, respectively. Then
\begin{equation}
 \mathsf D_F^{\mathrm{col}}[[i,j,k,l,n_s,m_r]]
 =
 \left[F^l_{i,j,k}\right]^{m_s}_{n_r}.
\end{equation}
The crossed \(r\) and \(s\) are intentional. Since the two lists of
admissible channels depend on \((i,j,k,l)\) and need not coincide, conversion
to direct fusion-path indices is block-dependent and cannot be expressed by
a fixed permutation of the key positions.

For the \(R\)-symbols,
\[
 \mathsf D_R^{\mathrm{col}}[[i,j,k,\alpha,\beta]]
 =
 \left[R^k_{i,j}\right]^\beta_\alpha .
\]
Without multiplicities, both \(R\)-symbol conventions use the key
\([i,j,k]\). The choice of convention changes only the dictionary encoding,
not the \(F\)- or \(R\)-matrices. Further details are given in the manual of
\textsc{TensorCategories.jl}.
\end{remark}

The following code shows how to define the Ising fusion category over a general field $K$. An element $a \in K$ with $a^2=2$ needs to be provided to define the associators (see the introduction).

\begin{tcolorbox}[title = \juliasymbol, breakable]
    \begin{juliaprompt}        
function ising_category(K::Ring, a::RingElem)
    C = six_j_category(K,["
    
    # Multiplication table of the Grothendieck ring
    M = zeros(Int,3,3,3)
    M[1,1,:] = [1,0,0]
    M[1,2,:] = [0,1,0]
    M[1,3,:] = [0,0,1]
    M[2,1,:] = [0,1,0]
    M[2,2,:] = [1,0,0]
    M[2,3,:] = [0,0,1]
    M[3,1,:] = [0,0,1]
    M[3,2,:] = [0,0,1]
    M[3,3,:] = [1,1,0]

    set_tensor_product!(C,M)

    # The associators
    set_associator!(C, 2, 3, 2, 3, matrix(K, 1, 1, [-1]))
    set_associator!(C ,3, 2, 3, 2, matrix(K, 1, 1, [-1]))
    set_associator!(C, 3, 3, 3, 3, inv(a)*matrix(K,[1 1; 1 -1]))

    set_one!(C,[1,0,0])
    set_spherical!(C, [K(1) for s in simples(C)])
    set_name!(C, "Ising fusion category")
    return C
end
    \end{juliaprompt}\vspace{-\baselineskip}
\end{tcolorbox}

We note that \verb+ising_category+ as above is already implemented in \textsc{TensorCategories.jl}, as well as many other examples of fusion categories. We can choose as $K$ the algebraic closure of $\mathbb{Q}$, which is also available in OSCAR via \texttt{algebraic\_closure(QQ)}. But we can also work over $K = \QQ(\sqrt{2})$, which is what we will do.

\begin{tcolorbox}[title = \juliasymbol, breakable]
    \begin{juliaprompt}
julia> K,r2 = quadratic_field(2)
(Real quadratic field defined by x^2 - 2, sqrt(2))

julia> C = ising_category(K,r2)
Ising fusion category

julia> a,b,c = simples(C)
3-element Vector{SixJObject}:
    \end{juliaprompt} \vspace{-\baselineskip}
\end{tcolorbox}

Let us now come to the proof Theorem \ref{thm:ising}. We compute the simple objects of the center by invoking the algorithm in Section \ref{sec:algorithm}. 

\begin{tcolorbox}[title = \juliasymbol]
    \begin{juliaprompt}
julia> Z = center(C)
Drinfeld center of Fusion Category with 3 simple objects

julia> S = simples(Z) 
5-element Vector{CenterObject}:
Central object: 
Central object: 
Central object: 
Central object: 2
Central object: 4
    \end{juliaprompt}\vspace{-\baselineskip}
\end{tcolorbox}     

We note that the \texttt{center} function at first only creates an empty structure for the center---the first real computational effort is in the \texttt{simples} function which computes the simple objects of the center. Notice that the computation involves randomized parts, so e.g., the ordering of the simples, the form of the precise half-braidings, and thus the basis of endomorphism spaces we compute below may differ from the exact output given here. Of course, the proof of Theorem \ref{thm:ising} is independent of this because we do the analogous computations with all simple objects of the center.

We can see that there are five non-isomorphic simple objects. We show that two of them are not split over $K$ and examine over which fields they will split. To do so we examine the endomorphism spaces. The central object lying over $2\cdot \chi$ will split if there is an endomorphism that is a zero-divisor, i.e. if there is a morphism with a non-trivial eigenvalue. Thus, we take a non-trivial endomorphism and consider the splitting field for its minimal polynomial.

\begin{tcolorbox}
    [title = \juliasymbol, breakable]
    \begin{juliaprompt}
julia> H = End(S[4]) 
Vector space of dimension 2 over Real quadratic field defined by x^2 - 2.

julia> minpoly.(basis(H)) # minimal polynomials of basis morphisms
2-element Vector{AbstractAlgebra.Generic.Poly{nf_elem}}:
 x - 1
 x^2 + 1//4
    \end{juliaprompt}\vspace{-\baselineskip}
\end{tcolorbox}

Hence, if we extend the base field $K$ to the splitting field $L$ of $x^2+\frac{1}{4}$ over $K$, which is the same as the splitting field of $x^2+1$ over $K$, so $L = \mathbb{Q}(\sqrt 2, i)$, it will split. Similarly, one can show that the fifth simple object $4 \cdot \chi$ decomposes under this extension.

\begin{tcolorbox}
    [title = \juliasymbol, breakable]
    \begin{juliaprompt}
julia> Kx,x = base_ring(C)[:x]
(Univariate polynomial ring in x over K, x)

julia> L,i = number_field(x^2+1, "i")
(Relative number field of degree 2 over K, i)

julia> L, = absolute_simple_field(L)
(Number field of degree 4 over QQ, Map: L -> relative number field)

julia> ZL = Z 
Drinfeld center of Fusion Category with 3 simple objects

julia> simples(ZL) 
7-element Vector{CenterObject}:
 Central object: 
 Central object: 
 Central object: 
 Central object: 
 Central object: 
 Central object: 2
 Central object: 2
\end{tcolorbox}

Repeat the splitting process one more time for the 6-th or 7-th simple.

\begin{tcolorbox}[title = \juliasymbol, breakable]
    \begin{juliaprompt}
julia> _,f = minpoly.(basis(End(ZL[6])))
2-element Vector{AbstractAlgebra.Generic.Poly{Hecke.NfRelElem{nf_elem}}}:
 x - 1
 x^2 + 1//24*x^3 - 1//8*x^2 - 5//24*x + 1//8

julia> M,a = number_field(f,"a")
(Relative number field of degree 2 over L, a)

julia> M, = absolute_simple_field(M)
(Number field of degree 8 over QQ, Map: M -> relative number field)

julia> simples(ZL 
9-element Vector{CenterObject}:
 Central object: 
 Central object: 
 Central object: 
 Central object: 
 Central object: X
 Central object: X
 Central object: X
 Central object: X
 Central object: 

julia> simplify(absolute_simple_field(M)[1])[1]
Number field with defining polynomial x^8 + 1
  over rational field
    \end{juliaprompt}\vspace{-\baselineskip}
\end{tcolorbox}

Thus, we see that the center of the Ising category splits over the cyclotomic field $\mathbb Q(\xi_{16})$ with $\xi_{16}$ a primitive 16-th root of unity.
Next, we compute the multiplication table and $S$-matrix of the center over $K$.

\begin{tcolorbox}[title = \juliasymbol, breakable]
    \begin{juliaprompt}
julia> print_multiplication_table(S, ["X$i" for i in eachindex(S)])
5
 "X1"  "X2"  "X3"            "X4"           "X5"
 "X2"  "X1"  "X3"            "X4"           "X5"
 "X3"  "X3"  "X1 
 "X4"  "X4"  "2
 "X5"  "X5"  "2

julia> smatrix(Z)
[        1            1    2    2    4*sqrt(2)]
[        1            1    2    2   -4*sqrt(2)]
[        2            2    0   -4            0]
[        2            2   -4    4            0]
[4*sqrt(2)   -4*sqrt(2)    0    0            0]
    \end{juliaprompt}\vspace{-\baselineskip}
\end{tcolorbox}

We can also investigate the actual half-braiding morphisms and endomorphism rings.

\begin{tcolorbox}[title = \juliasymbol, breakable]
    \begin{juliaprompt}
julia> half_braiding(Z[4])
3-element Vector{TensorCategories.SixJMorphism}:
    Morphism with
Domain: 2
Codomain: 2
Matrices: 0 by 0 empty matrix, [1 0; 0 1], 0 by 0 empty matrix
    Morphism with
Domain: 2
Codomain: 2
Matrices: [-1 0; 0 -1], 0 by 0 empty matrix, 0 by 0 empty matrix
    Morphism with
Domain: 2
Codomain: 2
Matrices: 0 by 0 empty matrix, 0 by 0 empty matrix, [0 1//2; -2 0]

julia> R = endomorphism_ring(Z[4])
Matrix algebra of dimension 2 over Real quadratic field defined by x^2 - 2

julia> basis(R)
2-element Vector{MatAlgebraElem{AbsSimpleNumFieldElem, AbstractAlgebra.Generic.MatSpaceElem{AbsSimpleNumFieldElem}}}:
 [1 0; 0 1]
 [0 -1//4; 1 0]
    \end{juliaprompt}\vspace{-\baselineskip}
\end{tcolorbox}

The full data for the endomorphism rings and half-braidings is given by 

\begin{align*}
    \End(2\chi) &= \left\langle \begin{pmatrix}1 & 0 \\ 0 & 1\end{pmatrix}, \begin{pmatrix} 0 & -1 \\ 4 & 0\end{pmatrix}\right\rangle_K \cong K[i] \\
    \End(4X) &= \left\langle{\tiny \begin{pmatrix} 1 & 0 & 0 & 0 \\ 0 & 1 & 0 & 0 \\ 0 & 0 & 1 & 0 \\ 0 & 0 & 0 & 1 \end{pmatrix}, \begin{pmatrix} 0 & -1 & 0 & 0 \\ 1 & 0 & 0 & 0 \\ 0 & 0 & 0 & 1 \\ 0 & 0 & -1 & 0 \end{pmatrix}, \begin{pmatrix} 0 & 0 & 0 & \frac{\sqrt2}{2} \\ 0 & 0 & \frac{\sqrt2}{2} & 0 \\ 1 & -1 & 0 & 0 \\ -1 & 1 & 0 & 0 \end{pmatrix}, \begin{pmatrix} 0 & 0 & \frac{\sqrt2}{2} & 0 \\ 0 & 0 & & 0 \frac{\sqrt2}{2} \\ 1 & 1 & 0 & 0 \\ 1 & -1 & 0 & 0 \end{pmatrix}  }\right\rangle_K \\ 
    &\cong K[\beta], \text{where}~\beta^4-\sqrt2\beta^2 + 1 = 0
\end{align*}

and 

\begin{align*}
    \gamma_{\overline{\mathbb 1}}(\chi) &= \id_{\chi}  & \gamma_{\overline{\mathbb 1}} &= -\id_{X} \\
    \gamma_{\mathbb 1 + \chi}(\chi) &= \id_{\mathbb 1} - \id_{\chi} & \gamma_{\mathbb 1 + \chi}(X) &= -\id_{2X} \\
    \gamma_{2\chi}(\chi) &= -\id{2\mathbb 1} & \gamma_{2\mathbb 1}(X) &= \begin{pmatrix}
        \frac{1}{2} & 0 \\ 0 & 2
    \end{pmatrix} \id_{2X} \\
    \gamma_{4X}(\chi) &= \begin{pmatrix}
        0 & -1 & 0 & 0 \\ 1 & 0 & 0 & 0 \\ 0 & 0 & 0 & 1 \\ 0 & 0 & -1 & 0
    \end{pmatrix}\id_{4X} & & \\
    \gamma_{4X}(X) &= \begin{pmatrix}
        0 & 0 & \frac{1}{2} & \frac{1}{2} \\ 0 & 0 & \frac{1}{2} & -\frac{1}{2} \\ \sqrt2 & 0 & 0 & 0 \\ 0 & -\sqrt2 & 0 & 0
    \end{pmatrix}\id_{4\mathbb 1} + \begin{pmatrix}
        0 & 0 & \frac{1}{2} & -\frac{1}{2} \\ 0 & 0 & -\frac{1}{2} & -\frac{1}{2} \\  0 & \sqrt2 & 0 & 0 \\ \sqrt2 & 0 & 0 & 0 
    \end{pmatrix}\id_{4\chi} \span \span.
\end{align*}

\vspace{\baselineskip}

\section{Centers of multiplicity-free fusion categories up to rank 5} \label{sec:AnyonCenters}

The fusion categories from the AnyonWiki are encoded by a sequence of integers $(r,m,n,i,j,k,l)$, where 
\begin{itemize}
    \item $r$ is the rank of the category,
    \item $m$ is the multiplicity,
    \item $n$ is the number of non-self-dual simple objects,
    \item $i$ is the index of the fusion ring,
    \item $j$ is the index of the associator,
    \item $k$ is the index of the braiding and
    \item $l$ is the index of the pivotal structure.
\end{itemize}

The details on how the indices are determined are explained in \cite{vercleyen2024lowrankmultiplicityfreefusioncategories} or on the webpage of the AnyonWiki \cite{AnyonWiki}. The following example shows essentially the steps we did with every category of the AnyonWiki. 

\begin{tcolorbox}[title = \juliasymbol, breakable]
    \begin{juliaprompt}
julia> C = anyonwiki(3,1,0,3,1,1,1);

julia> Z = center(C);

julia> simples(Z);

julia> Z2 = split(Z)[1];

julia> Z3 = six_j_category(Z2);

julia> save_fusion_category(Z3,"/tmp","saved_cat");

julia> Z4 = load_fusion_category("/tmp/saved_cat");
    \end{juliaprompt}\vspace{-\baselineskip}
\end{tcolorbox}

The function \verb+six_j_category+ computes the $F$-symbols (algebraically) using \Autoref{alg:F_symbols} and, as the example shows, we can also store and load this data. This is also how we created the database. 

The following example shows how to access our database of centers. This particular category has four simple objects of which two are self-dual and two are not. The Grothendieck ring has Frobenius--Perron dimension 13.6569, which is the largest possible for multiplicity one and rank four. The center then is a modular fusion category of dimension 186.510, has 14 simple objects and multiplicity two.

\begin{tcolorbox}[title = \juliasymbol, breakable]
    \begin{juliaprompt}
julia> C = anyonwiki(4,1,2,4,1,0,1)
Fusion Category 4_1_2_4_1_0_1

julia> Z = anyonwiki_center(4,1,2,4,1,0,1)
Skeletonization of Drinfeld center of Fusion Category 4_1_2_4_1_0_1

julia> simples(C)
4-element Vector{SixJObject}:
 X2
 X3
 X4

julia> print_multiplication_table(C)
4
 "
 "X2"  "
 "X3"  "X4"  "X2 
 "X4"  "X3"  "

 julia> simples(Z)
14-element Vector{SixJObject}:

# Print the multiplication table of the center with the simple objects named X1,...,X14
julia> print_multiplication_table(multiplication_table(Z))
14
 "X1"   "X2"   "X3"               
 "X2"   "X1"   "X5"                    "X13"
 "X3"   "X5"   "X4 
 "X4"   "X6"   "X2 
 "X5"   "X3"   "X6 
 "X6"   "X4"   "X1 
 "X7"   "X9"   "X4 
 "X8"   "X10"  "X9 
 "X9"   "X7"   "X6 
 "X10"  "X8"   "X7 
 "X11"  "X12"  "X5 
 "X12"  "X11"  "X3 
 "X13"  "X14"  "X8 
 "X14"  "X13"  "X10 

# Partial view of the S-matrix
julia> view(smatrix(Z),1:4,1:4)
[                1                  -1      -_a^3 - _a^2 - _a      -_a^3 - _a^2 - _a]
[               -1                  -1      -_a^3 - _a^2 - _a       _a^3 + _a^2 + _a]
[-_a^3 - _a^2 - _a   -_a^3 - _a^2 - _a       -_a^2 - 2*_a - 2   -2*_a^3 - 2*_a^2 + 1]
[-_a^3 - _a^2 - _a    _a^3 + _a^2 + _a   -2*_a^3 - 2*_a^2 + 1       -_a^2 - 2*_a - 2]

# The T-matrix
julia> diagonal(tmatrix(Z))
14-element Vector{AbsSimpleNumFieldElem}:
 1
 -_a^2
 _a^3 - _a^2 - _a
 _a^3 - _a - 1
 -_a^3 - _a + 1
 _a^3 + _a^2 + _a
 1//2*_a^3 - 3//2*_a^2 - 3//2*_a - 1//2
 1//2*_a^3 + 1//2*_a^2 - 3//2*_a + 1//2
 -3//2*_a^3 - 1//2*_a^2 - 1//2*_a + 3//2
 -3//2*_a^3 + 1//2*_a^2 - 1//2*_a - 1//2
 -_a^3
 _a
 2*_a^3 - 2*_a + 3
 2*_a^3 - 3*_a^2 + 2*_a


    \end{juliaprompt}\vspace{-\baselineskip}
\end{tcolorbox}

Last, we want to investigate the $F$-symbols and $R$-symbols. This can be done either symbolically or numerically. To view numerical symbols we need to specify an embedding into the complex numbers. For the fusion categories from the AnyonWiki and their centers this was already done and is stored in our database \cite{TensorCategoriesDatabase}. 

\begin{tcolorbox}[title = \juliasymbol, breakable]
    \begin{juliaprompt}
julia> F_symbols(Z)
Dict{Vector{Int64}, AbsSimpleNumFieldElem} with 74880 entries:
  [10, 14, 4, 2, 7, 1, 1, 3, 1, 1]     => -_a^3
  [11, 10, 8, 5, 7, 1, 1, 6, 1, 1]     => _a^3 + 1//2*_a^2 - 1//2*_a - 1
  [10, 6, 8, 9, 13, 1, 1, 14, 1, 1]    => _a^3 + 1//2*_a^2 - 1//2
  [14, 13, 13, 14, 10, 1, 1, 11, 1, 1] => -35//17*_a^3 + 70//17*_a^2 - 72//17*_a + 25//17
  [8, 6, 11, 14, 13, 1, 1, 13, 1, 1]   => -1//12*_a^3 - 7//12*_a^2 + 3//4*_a - 7//12

julia> R_symbols(Z)
Dict{Vector{Int64}, AbsSimpleNumFieldElem} with 592 entries:
  [7, 3, 4, 1, 1]    => -_a^2
  [4, 8, 14, 1, 1]   => -_a^2
  [11, 1, 11, 1, 1]  => 1
  [11, 13, 14, 2, 2] => _a^3 - _a^2
  [11, 11, 10, 1, 1] => -_a^2
  [13, 13, 12, 2, 1] => 3//34*_a^3 + 11//34*_a^2 - 5//34*_a + 27//34

julia> num_F_symbs = numeric_F_symbols(Z, precision = 8) # precision in bits
Dict{Vector{Int64}, AcbFieldElem} with 74880 entries:
  [7, 12, 10, 14, 10, 1, 1, 11, 1, 1]  => [-0.414 +/- 3.41e-4] + [0.414 +/- 4.02e-4]*im
  [10, 14, 4, 2, 7, 1, 1, 3, 1, 1]     => [0.707 +/- 1.49e-4] + [-0.707 +/- 1.49e-4]*im
  [10, 13, 7, 11, 11, 1, 1, 13, 1, 1]  => [-0.172 +/- 5.13e-4] + [+/- 2.66e-4]*im
  [11, 14, 13, 13, 3, 1, 1, 10, 1, 1]  => [-0.293 +/- 1.49e-4] + [-0.707 +/- 1.49e-4]*im
  [13, 8, 13, 14, 8, 1, 1, 11, 1, 1]   => [-0.14 +/- 4.84e-3] + [-1.16 +/- 4.52e-3]*im
  [11, 10, 8, 5, 7, 1, 1, 6, 1, 1]     => [-2.061 +/- 5.09e-4] + [0.853 +/- 6.25e-4]*im

julia> numeric_R_symbols(Z, precision = 8)
Dict{Vector{Int64}, AcbFieldElem} with 592 entries:
  [7, 3, 4, 1, 1]    => [+/- 4.32e-5] + [-1.00 +/- 1.05e-4]*im
  [4, 8, 14, 1, 1]   => [+/- 4.32e-5] + [-1.00 +/- 1.05e-4]*im
  [11, 1, 11, 1, 1]  => 1.00
  [11, 13, 14, 2, 2] => [-0.707 +/- 1.93e-4] + [-0.293 +/- 2.54e-4]*im
  [11, 11, 10, 1, 1] => [+/- 4.32e-5] + [-1.00 +/- 1.05e-4]*im
  [13, 13, 12, 2, 1] => [0.628 +/- 3.44e-4] + [0.2819 +/- 6.86e-5]*im
    \end{juliaprompt}\vspace{-\baselineskip}
\end{tcolorbox} 

The numeric $F$-symbols can be stored and loaded as in the following example:

\begin{tcolorbox}[title = \juliasymbol, breakable]
    \begin{juliaprompt}
julia> numeric_symbols_to_csv("numeric_F_symbols", num_F_symbs);

julia> load_numeric_fusion_category("numeric_F_symbols")
Fusion Category with 9 simple objects
    \end{juliaprompt}\vspace{-\baselineskip}
\end{tcolorbox}

\subsection{Some statistics}

In the following table we give some numeric results for the centers of all multiplicity-free fusion categories up to rank five. We list the rank, the multiplicity, the Frobenius--Perron dimension as well as the largest simple subcategory of the centers. Note that the indices for the braiding and pivotal structures are not relevant for the structure of the centers as a braided fusion category. Therefore, for readability, we only list a representative for the first five indices.

There are some interesting observations. For example, there are only twelve categories whose center has multiplicity two and no center is simple. 

\def\arraystretch{1.5}

\begin{longtable}{c||c|c|c|c|c}
\textbf{Label} & \textbf{Rank} & \textbf{Multiplicity} & \textbf{FPdim} & \multicolumn{2}{c}{\textbf{Max. simple subcategory}} \\ 
    &   &   &   &  \textbf{Rank} & \textbf{FPdim} \\ \hline
        \AnyonCode{1,1,0,1,1,a,b} & 1 & 1 & 1.00 & 1 & 1.00 \\ 
\AnyonCode{2,1,0,1,1,a,b} & 4 & 1 & 4.00 & 2 & 2.00 \\ 
\AnyonCode{2,1,0,1,2,a,b} & 4 & 1 & 4.00 & 2 & 2.00 \\ 
\AnyonCode{2,1,0,2,1,a,b} & 4 & 1 & 13.09 & 2 & 3.62 \\ 
\AnyonCode{2,1,0,2,2,a,b} & 4 & 1 & 13.09 & 2 & 3.62 \\ 
\AnyonCode{3,1,0,1,1,a,b} & 9 & 1 & 16.00 & 2 & 2.00 \\ 
\AnyonCode{3,1,0,1,2,a,b} & 9 & 1 & 16.00 & 2 & 2.00 \\ 
\AnyonCode{3,1,0,2,1,a,b} & 8 & 1 & 36.00 & 2 & 2.00 \\ 
\AnyonCode{3,1,0,2,2,a,b} & 8 & 1 & 36.00 & 2 & 2.00 \\ 
\AnyonCode{3,1,0,2,3,a,b} & 8 & 1 & 36.00 & 2 & 2.00 \\ 
\AnyonCode{3,1,0,3,1,a,b} & 9 & 1 & 86.41 & 3 & 9.30 \\ 
\AnyonCode{3,1,0,3,2,a,b} & 9 & 1 & 86.41 & 3 & 9.30 \\ 
\AnyonCode{3,1,0,3,3,a,b} & 9 & 1 & 86.41 & 3 & 9.30 \\ 
\AnyonCode{3,1,2,1,1,a,b} & 9 & 1 & 9.00 & 3 & 3.00 \\ 
\AnyonCode{3,1,2,1,2,a,b} & 9 & 1 & 9.00 & 3 & 3.00 \\ 
\AnyonCode{3,1,2,1,3,a,b} & 9 & 1 & 9.00 & 3 & 3.00 \\ 
\AnyonCode{4,1,0,1,1,a,b} & 16 & 1 & 16.00 & 2 & 2.00 \\ 
\AnyonCode{4,1,0,1,2,a,b} & 16 & 1 & 16.00 & 2 & 2.00 \\ 
\AnyonCode{4,1,0,1,3,a,b} & 16 & 1 & 16.00 & 2 & 2.00 \\ 
\AnyonCode{4,1,0,1,4,a,b} & 16 & 1 & 16.00 & 2 & 2.00 \\ 
\AnyonCode{4,1,0,2,1,a,b} & 16 & 1 & 52.36 & 2 & 3.62 \\ 
\AnyonCode{4,1,0,2,2,a,b} & 16 & 1 & 52.36 & 2 & 3.62 \\ 
\AnyonCode{4,1,0,2,3,a,b} & 16 & 1 & 52.36 & 2 & 3.62 \\ 
\AnyonCode{4,1,0,2,4,a,b} & 16 & 1 & 52.36 & 2 & 3.62 \\ 
\AnyonCode{4,1,0,3,1,a,b} & 16 & 1 & 100.00 & 2 & 2.00 \\ 
\AnyonCode{4,1,0,3,2,a,b} & 16 & 1 & 100.00 & 2 & 2.00 \\ 
\AnyonCode{4,1,0,3,3,a,b} & 16 & 1 & 100.00 & 2 & 2.00 \\ 
\AnyonCode{4,1,0,4,1,a,b} & 14 & 2 & 186.51 & 2 & 2.00 \\ 
\AnyonCode{4,1,0,4,2,a,b} & 14 & 2 & 186.51 & 2 & 2.00 \\ 
\AnyonCode{4,1,0,5,1,a,b} & 16 & 1 & 171.35 & 2 & 3.62 \\ 
\AnyonCode{4,1,0,5,2,a,b} & 16 & 1 & 171.35 & 2 & 3.62 \\ 
\AnyonCode{4,1,0,5,3,a,b} & 16 & 1 & 171.35 & 2 & 3.62 \\ 
\AnyonCode{4,1,0,6,1,a,b} & 16 & 1 & 369.96 & 4 & 19.23 \\ 
\AnyonCode{4,1,0,6,2,a,b} & 16 & 1 & 369.96 & 4 & 19.23 \\ 
\AnyonCode{4,1,0,6,3,a,b} & 16 & 1 & 369.96 & 4 & 19.23 \\ 
\AnyonCode{4,1,2,1,1,a,b} & 16 & 1 & 16.00 & 2 & 2.00 \\ 
\AnyonCode{4,1,2,1,2,a,b} & 16 & 1 & 16.00 & 2 & 2.00 \\ 
\AnyonCode{4,1,2,1,3,a,b} & 16 & 1 & 16.00 & 2 & 2.00 \\ 
\AnyonCode{4,1,2,1,4,a,b} & 16 & 1 & 16.00 & 2 & 2.00 \\ 
\AnyonCode{4,1,2,2,1,a,b} & 15 & 1 & 36.00 & 3 & 3.00 \\ 
\AnyonCode{4,1,2,2,2,a,b} & 15 & 1 & 36.00 & 3 & 3.00 \\ 
\AnyonCode{4,1,2,2,3,a,b} & 15 & 1 & 36.00 & 3 & 3.00 \\ 
\AnyonCode{4,1,2,2,4,a,b} & 15 & 1 & 36.00 & 3 & 3.00 \\ 
\AnyonCode{4,1,2,4,1,a,b} & 14 & 2 & 186.51 & 7 & 93.25 \\ 
\AnyonCode{4,1,2,4,2,a,b} & 14 & 2 & 186.51 & 7 & 93.25 \\ 
\AnyonCode{4,1,2,4,3,a,b} & 14 & 2 & 186.51 & 7 & 93.25 \\ 
\AnyonCode{4,1,2,4,4,a,b} & 14 & 2 & 186.51 & 7 & 93.25 \\ 
\AnyonCode{5,1,0,1,1,a,b} & 22 & 1 & 64.00 & 2 & 2.00 \\ 
\AnyonCode{5,1,0,1,2,a,b} & 22 & 1 & 64.00 & 2 & 2.00 \\ 
\AnyonCode{5,1,0,1,3,a,b} & 22 & 1 & 64.00 & 2 & 2.00 \\ 
\AnyonCode{5,1,0,1,4,a,b} & 22 & 1 & 64.00 & 2 & 2.00 \\ 
\AnyonCode{5,1,0,3,1,a,b} & 25 & 1 & 144.00 & 2 & 2.00 \\ 
\AnyonCode{5,1,0,3,2,a,b} & 25 & 1 & 144.00 & 2 & 2.00 \\ 
\AnyonCode{5,1,0,4,1,a,b} & 28 & 1 & 196.00 & 2 & 2.00 \\ 
\AnyonCode{5,1,0,4,2,a,b} & 28 & 1 & 196.00 & 2 & 2.00 \\ 
\AnyonCode{5,1,0,4,3,a,b} & 28 & 1 & 196.00 & 2 & 2.00 \\ 
\AnyonCode{5,1,0,6,1,a,b} & 21 & 2 & 576.00 & 2 & 2.00 \\ 
\AnyonCode{5,1,0,6,2,a,b} & 21 & 2 & 576.00 & 2 & 2.00 \\ 
\AnyonCode{5,1,0,7,1,a,b} & 22 & 2 & 685.41 & 2 & 2.00 \\ 
\AnyonCode{5,1,0,7,2,a,b} & 22 & 2 & 685.41 & 2 & 2.00 \\ 
\AnyonCode{5,1,0,10,1,a,b} & 25 & 1 & 1200.38 & 5 & 34.65 \\ 
\AnyonCode{5,1,0,10,2,a,b} & 25 & 1 & 1200.38 & 5 & 34.65 \\ 
\AnyonCode{5,1,0,10,3,a,b} & 25 & 1 & 1200.38 & 5 & 34.65 \\ 
\AnyonCode{5,1,0,10,4,a,b} & 25 & 1 & 1200.38 & 5 & 34.65 \\ 
\AnyonCode{5,1,0,10,5,a,b} & 25 & 1 & 1200.38 & 5 & 34.65 \\ 
\AnyonCode{5,1,2,1,1,a,b} & 22 & 1 & 64.00 & 2 & 2.00 \\ 
\AnyonCode{5,1,2,1,2,a,b} & 22 & 1 & 64.00 & 2 & 2.00 \\ 
\AnyonCode{5,1,2,1,3,a,b} & 22 & 1 & 64.00 & 2 & 2.00 \\ 
\AnyonCode{5,1,2,1,4,a,b} & 22 & 1 & 64.00 & 2 & 2.00 \\ 
\AnyonCode{5,1,2,3,1,a,b} & 25 & 1 & 144.00 & 2 & 2.00 \\ 
\AnyonCode{5,1,2,3,2,a,b} & 25 & 1 & 144.00 & 2 & 2.00 \\ 
\AnyonCode{5,1,2,4,1,a,b} & 21 & 2 & 576.00 & 2 & 2.00 \\ 
\AnyonCode{5,1,2,4,2,a,b} & 21 & 2 & 576.00 & 2 & 2.00 \\ 
\AnyonCode{5,1,4,1,1,a,b} & 25 & 1 & 25.00 & 5 & 5.00 \\ 
\AnyonCode{5,1,4,1,2,a,b} & 25 & 1 & 25.00 & 5 & 5.00 \\ 
\AnyonCode{5,1,4,1,3,a,b} & 25 & 1 & 25.00 & 5 & 5.00 
\end{longtable}

\vspace{\baselineskip}

\subsection{Unitary categories}

For applications in mathematical physics, \emph{unitary} fusion categories play an important role. These are fusion categories over $\mathbb{C}$ with the property that there exist bases for the multiplicity spaces $H_{i,j}^k$, a so-called \emph{gauge}, in which the $F$-matrices are \emph{unitary} complex matrices. For our purpose we will call a fusion category given by $F$-matrices \emph{unitary} if the corresponding $F$-matrices are all unitary. Given a unitary fusion category, its center again is unitary by \cite[Theorem 6.4]{subfactors}. In \textsc{TensorCategories.jl} we have implemented functionality to work with unitary fusion categories, and the computed center of a unitary fusion category will be unitary as well. The following code gives an example:

\begin{tcolorbox}[title = \juliasymbol, breakable]
\begin{juliaprompt}
julia> C = anyonwiki(QQBarField(),3,1,0,1,2,1,1) # Unitary Ising
Fusion Category 3_1_0_1_2_1_1

julia> is_unitary(C)
true

julia> Z = center(C)
Drinfeld center of Fusion Category 3_1_0_1_2_1_1

julia> simples(Z)
9-element Vector{CenterObject}:
 Central object: 
 Central object: X2
 Central object: X3
 Central object: X2
 Central object: 
 Central object: X3
 Central object: X3
 Central object: X3
 Central object: 

julia> is_unitary.(half_braiding(Z[9]))
3-element BitVector:
 1
 1
 1

julia> Z2 = six_j_category(Z)
Skeletonization of Drinfeld center of Fusion Category 3_1_0_1_2_1_1

julia> is_unitary(associator(Z2[[9,9,9]]...))
true
\end{juliaprompt}
\end{tcolorbox}

\subsection*{Data availability statement}
Our software package \textsc{TensorCategories.jl} developed as part of this paper and used to produce the results in this paper is entirely open-source and publicly available in the repository \url{https://github.com/
TensorCategories/TensorCategories.jl}. The version used for this paper (v0.6.0) is additionally archived at \url{https://doi.org/10.5281/zenodo.21154104}. Our software package is an officially registered Julia package, and it works under any operating system as described in the OSCAR installation guide at \url{https://www.oscar-system.org/install/}. The database with $F$-symbols and $R$-symbols of centers we computed is available at \url{https://github.com/TensorCategories/TensorCategoriesDatabase} and it is additionally archived at \url{https://doi.org/10.5281/zenodo.20825407}. All data can be reproduced with a script in the folder \verb+computations/center_paper+ of the source repository, which also contains a table of timings for each category. The data we used from the Anyonica package is available at \url{https://doi.org/10.5281/zenodo.15920400} and the AnyonWiki is available at \url{https://doi.org/10.5281/zenodo.15863113}.

\subsection*{Conflict of interest statement}

The authors have no competing interests to declare that are relevant to the content of this article.

\end{document}